\documentclass[a4paper,11pt,reqno]{amsart}
\usepackage{amsmath,amsfonts,amsthm,amssymb,color}
\usepackage[T1]{fontenc}
\usepackage{pdfsync}
\usepackage{csquotes}
\usepackage{graphicx}
\usepackage{pstricks}
\usepackage{lmodern}

\newcommand{\be}{\beta}

\newcommand{\id}{\mbox{Id}}

\newcommand{\C}{\mathbb C}

\newcommand{\R}{\mathbb R}

\newcommand{\ca}{\mathcal A}

\newcommand{\cac}{\mathcal C}

\newcommand{\ce}{\mathcal E}

\newcommand{\cj}{\mathcal J}
\newcommand{\cl}{\mathcal L}

\newcommand{\cs}{\mathcal S}

\newcommand{\al}{\alpha}
\newcommand{\der}{\delta}

\newcommand{\ga}{\gamma}
\newcommand{\ka}{\kappa}
\newcommand{\la}{\lambda}

\newcommand{\si}{\sigma}

\newcommand{\vp}{\varphi}

\newtheorem{theorem}{Theorem}[section]

\newtheorem{corollary}[theorem]{Corollary}

\newtheorem{definition}[theorem]{Definition}

\newtheorem{lemma}[theorem]{Lemma}
\newtheorem{notation}[theorem]{Notation}

\newtheorem{proposition}[theorem]{Proposition}

\theoremstyle{remark}
\newtheorem{remark}[theorem]{Remark}

\date{\today}

\begin{document}

\makeatletter
\def\@settitle{\begin{center}%
  \baselineskip14\p@\relax
    \normalfont\LARGE
\@title
  \end{center}%
}
\makeatother

\title{Integration with respect to the non-commutative fractional Brownian motion}

\author{Aur\'elien Deya}
\address[A. Deya]{Institut Elie Cartan, University of Lorraine
B.P. 239, 54506 Vandoeuvre-l\`es-Nancy, Cedex
France}
\email{aurelien.deya@univ-lorraine.fr}

\author{Ren\'e Schott}
\address[R. Schott]{Institut Elie Cartan, University of Lorraine
B.P. 239, 54506 Vandoeuvre-l\`es-Nancy, Cedex
France}
\email{rene.schott@univ-lorraine.fr}



\begin{abstract}


We study the issue of integration with respect to the non-commutative fractional Brownian motion, that is the analog of the standard fractional Brownian in a non-commutative probability setting.

\smallskip

When the Hurst index $H$ of the process is stricly larger than $1/2$, integration can be handled through the so-called Young procedure. The situation where $H=1/2$ corresponds to the specific free case, for which an It{\^o}-type approach is known to be possible.

\smallskip

When $H<1/2$, rough-path-type techniques must come into the picture, which, from a theoretical point of view, involves the use of some a-priori-defined L{\'e}vy area process. We show that such an object can indeed be \enquote{canonically} constructed for any $H\in (\frac14,\frac12)$. Finally, when $H\leq 1/4$, we exhibit a similar non-convergence phenomenon as for the non-diagonal entries of the (classical) L{\'e}vy area above the standard fractional Brownian.
\end{abstract} 

\maketitle

\vspace{-0.5cm}

\section{Introduction: the non-commutative fractional Brownian motion}\label{sec:introduction}

In classical probability theory, the fractional Brownian motion (fBm in the sequel) is considered as one of the most natural extensions of the standard Wiener process. From a modelling point of view, fractional noises offer the possibility to account for long-range dependency phenomenon, which easily explains their large success in various domains ranging from biological sciences to mathematical finance. The literature related to fBm now comprises thousands of publications, and we will only refer here to the nice survey \cite{nourdin-fbm}, which offers an overview on some of the most interesting aspects of this specific Gaussian process.

\smallskip

Unfortunately, when it comes to stochastic integration, the long-range dependence of the fBm turns into a major drawback and is known to be the source of important difficulties. In particular, fBm does not satisfy the martingale property, which rules out the possibility to use It{\^o} theory as a way to investigate the integration problem. More or less sophisticated alternative procedures must then come into the picture, based on either a \enquote{stochastic} approach (Malliavin calculus, Skorohod integral) or a \enquote{pathwise} strategy (Young integral, rough paths theory). Here again, any attempt to draw up an exhaustive list of the publications related to the fractional integration issue would be vain, and we will only quote the recent survey \cite{friz-hairer} about pathwise methods - the most efficient approach so far. 

\

In this first section, and as an introduction to the subsequent investigations, we propose to recall that the above fundamental objects (Wiener process, fBm, Gaussian processes) all admit immediate analogs in the so-called \textit{non-commutative probability setting}, the main framework of our study. Let us first recall, at a very general level, that non-commutative probability theory has received a lot of attention since the late 80's and the pathbreaking results of Voiculescu on large random matrices \cite{voiculescu}. Based on Voiculescu's results (together with subsequent extensions), the non-commutative paradigm can somehow be seen as a privileged formalism to study the asymptotic behaviour of standard classes of random matrices growing to infinity.

\smallskip

The rigourous presentation of this setting - which will prevail throughout the study - goes as follows, along the terminology of \cite{nica-speicher}:  

\begin{definition}
We call a non-commutative probability space any pair $(\ca,\vp)$ where:

\smallskip

\noindent
$(i)$ $\ca$ is a unital algebra over $\C$ endowed with an antilinear $\ast$-operation $X\mapsto X^\ast$ such that $(X^\ast)^\ast=X$ and $(XY)^\ast=Y^\ast X^\ast$ for all $X,Y\in \ca$. In addition, there exists a norm $\| .\| : \ca \to [0,\infty[$ which makes $\ca$ a Banach space, and such that for all $X,Y\in \ca$, $\| XY\|\leq \|X\| \|Y\|$ and $\|X^\ast X\|=\|X\|^2$.

\smallskip

\noindent
$(ii)$ $\vp:\ca \to \C$ is a linear functional such that $\vp(1)=1$, $\vp(XY)=\vp(YX)$, $\vp(X^\ast X)\geq 0$ for all $X,Y\in \ca$, and $\vp(X^\ast X)=0 \Leftrightarrow X=0$. We call $\vp$ the trace of the space.

\smallskip

In this setting, we will call any element $X\in \ca$ a non-commutative random variable, and any path $X_.:[0,T] \to \ca$ a non-commutative process. 
\end{definition}

A fundamental feature of any such non-commutative probability space lies in the close (hidden) link between the norm $\|.\|$ in item $(i)$ and the trace $\vp$ in item $(ii)$. Namely, for any $X\in \ca$, it can be shown on the one hand (see \cite[Proposition 3.8]{nica-speicher}) that 
\begin{equation}\label{trace-norm-1}
|\vp(X)|\leq \|X\| \ ,
\end{equation}
and, even more strikingly, one has (see \cite[Proposition 3.17]{nica-speicher})
\begin{equation}\label{trace-norm-2}
\|X\|:=\lim_{r\to \infty} \vp\big(\big(X X^\ast\big)^r\big)^{\frac{1}{2r}} \ .
\end{equation}
Thus, the trace $\vp$ can somehow be seen as the \enquote{expectation} in this setting and, along this analogy, the norm in $\ca$ can then be recovered as the \enquote{$L^\infty$-norm}. Besides, using standard spectral properties, we can provide some partial correspondance between the non-commutative framework and more classical probabilistic objects: namely, with any \textit{self-adjoint} element $X\in \ca$, we can associate a unique probability measure $\nu_X$ (called the \textit{law} of $X$) such that for every $k\geq 0$, 
\begin{equation}\label{law}
\vp\big(X^k\big)=\int_{\R} x^k \nu_X(dx) \ .
\end{equation}
Note that due to the (possible) non-commutativity of $\ca$, there is no hope to raise such a correspondance at the level of vectors (or processes), and to associate any non-commutative random vector $(X_1,\ldots,X_d)$ with a measure on $\R^d$. Instead, we consider that the \enquote{stochastic} dynamics of a given family $\{X_i\}_{i\in i}$ of non-commutative random variables is fully accounted by the set of its joint moments
$$\vp\big( X_{i_1}\cdots X_{i_r}\big) \quad , \quad r\geq 1 \ , i_1,\ldots,i_r \in I \ .$$

\

With these preliminaries in mind, let us turn to the presentation of the non-commutative process at the center of our study: the \textit{non-commutative fractional Brownian motion} (NC-fBm in the sequel). Just as the standard fBm is an example of a Gaussian process, the NC-fBm is part of a well-identified and important class of non-commutative processes, the so-called \textit{semicircular processes}. For a clear description of this family, let us introduce the following notation, borrowed from \cite{nica-speicher}, and that we will extensively use in our study:

\begin{notation}
Given elements $X_1,\ldots,X_{2m} \in \ca$ and a pairing $\pi$ of $\{1,\ldots,2m\}$ (that is, a partition of $\{1,\ldots,2m\}$ into $m$ disjoint subsets, each of cardinality $2$), we set
\begin{equation}\label{notation:vp-pi}
\ka_\pi\big( X_1,\ldots,X_{2m}\big):=\prod_{\{p,q\}\in \pi} \vp\big( X_{p}X_{q}\big) \ .
\end{equation}

Also, we denote by $NC_2(\{1,\ldots,2m\})$ or $NC_2(2m)$ the subset of non-crossing pairings of $\{1,\ldots,2m\}$, that is the subset of pairings $\pi$ of $\{1,\ldots,2m\}$ for which there is no elements $\{p_1,q_1\},\{p_2,q_2\}\in \pi$ with $p_1 < p_2 <q_1 <q_2$.
\end{notation}

\begin{definition}\label{def:semicircular-family}
With the above notation, we call a semicircular family in a non-commutative probability space $(\ca,\vp)$ any collection $\{X_i\}_{i \in I}$ of self-adjoint elements in $\ca$ such that, for every even integer $r\geq $1 and
all $i_1,\ldots,i_r \in I$, one has
\begin{equation}\label{form-wick}
\vp\big( X_{i_1}\cdots X_{i_r}\big)=\sum_{\pi \in NC_2(\{1,\ldots,r\})} \ka_\pi\big( X_{i_1},\ldots,X_{i_r}\big) \ ,
\end{equation}
and $\vp\big( X_{i_1}\cdots X_{i_r}\big)=0$ whenever $r$ is an odd integer.
\end{definition}

The law of a semicircular family (i.e., the set of its joint moments) is thus governed by what can be seen as a \enquote{non-commutative Wick formula}, obtained by restricting the usual sum to the sole non-crossing pairings. In particular, this law is fully determined by the set of the covariances $\{\vp( X_i X_j), \, i,j\in I\}$ of the family.

\smallskip

It is worth mentioning here that this analogy with the classical Gaussian processes extends through a fundamental central-limit property. In brief, semicircular families also appear as the universal limit (in the sense of the joint moments) of the renormalized sum of a sequence of \enquote{independent} NC-random families, where the notion of independence must be understood in some specific sense, the so-called free sense (see \cite[Theorem 8.17]{nica-speicher} for a complete statement).

\smallskip

As an immediate consequence of Definition \ref{def:semicircular-family} (and as an additional similarity with the Gaussian model), observe that the semicircular property is stable through linear real transformations. Let us label this elementary result for further references:
\begin{lemma}
Let $(X_1,\ldots,X_k)$ ($k\geq 1$) be a semicircular vector and $M$ a $k \times \ell$-matrix ($\ell \geq 1$) with real entries. Then $Y:=MX$ is a semicircular vector as well.
\end{lemma}

It is also worth recalling the non-expert reader that the semicircular property is named after the probability distribution it generates, when considering single random variables:

\begin{lemma}\label{lem:law-semicircular}
The law $\nu_X$ (in the sense of (\ref{law})) of a semicircular random variable $X$ is the semicircular distribution of variance $\si^2:=\vp(X^2)$, that is $\nu_X$ is the probability measure with density given by
$$p_{\si^2}(x):=\frac{1}{2\pi \si^2} \sqrt{4\si^2-x^2} \, \mathbf{1}_{\{|x|\leq 2\si\}} \ .$$
\end{lemma}

\

Here is finally the definition of the process (or rather the family of processes) at the center of our study: 
\begin{definition}\label{defi:q-bm}
In a NC-probability space $(\ca,\vp)$, and for every $H\in (0,1)$, we call a non-commutative fractional Brownian motion (NC-fBm) of Hurst index $H$ any semicircular family $\{X_t\}_{t\geq 0}$ in $(\ca,\vp)$ with covariance function given by the formula
\begin{equation}\label{cova-NC-fBm}
\vp\big( X_{s}X_{t}\big)=\frac12 \big\{s^{2H}+t^{2H}-|t-s|^{2H}\big\} \ .
\end{equation}
In particular, for every $t\geq 0$, $X_t$ is distributed along the semicircular distribution of variance $t^{2H}$.
\end{definition}

This definition should of course not be a surprise to any reader familiar with the definition of the standard fBm (in the classical setting). Formula (\ref{cova-NC-fBm}) is indeed nothing but the covariance function of the latter process. Lifting the formula to the level of the processes (using (\ref{form-wick})) gives rise to very different dynamics though, as can immediately be seen from Lemma \ref{lem:law-semicircular}. 

\smallskip

Note that for every fixed Hurst index $H\in (0,1)$, the existence of such a NC-fBm in some NC-probability space $(\ca,\vp)$ follows (for instance) from the general semicircular constructions of \cite{q-gauss} in the free Fock space.

\smallskip

Just as in the classical setting, the situation where $H=\frac12$ is very specific: the resulting non-commutative process here corresponds to the celebrated free Brownian motion, that is the non-commutative counterpart of the standard Wiener process. In this case, the disjoint increments of $X$ are known to satisfy the above-mentioned free independence property, a powerful tool at the very core of the results of \cite{biane-speicher} on non-commutative stochastic integration (see Section \ref{subsec:free-case} below for a few additional comments on this situation).

\smallskip

Unfortunately, as soon as $H\neq \frac12$ - which is the condition we have in mind here -, it can be shown that the disjoint increments of the NC-fBm are no longer freely independent. Free independence will thus not play any role in our analysis, and for this reason, we refrain from giving the exact definition of this property.

\

Before going further with the integration problem, let us mention the fact that this is not the first occurrence of the process in the literature. In \cite{nourdin-taqqu}, Nourdin and Taqqu have shown that the NC-fBm arises as the limit of natural sums constructed from a given stationary semicircular process. For the sake of conciseness, we cannot give a full account on their results (which relies in particular on the consideration of the Tchebycheff polynomials), but let us report the following simplified statement as an illustration of such asymptotic properties:

\begin{proposition}{\cite[Proposition 8.3]{nourdin-fbm}}
In a NC-probability space $(\ca,\vp)$, let $(U_k)_{k\geq 1}$ be a semicircular sequence such that $\vp(U_k)=0$, $\vp(U_k^2)=1$ and with stationary covariance (i.e., $\vp(X_kX_\ell)=\rho(k-\ell)$) satisfying 
$$\sum_{k,\ell=1}^n \rho(k-\ell)\sim C n^{2H} L(n) \quad  \text{as} \ n\to \infty\ ,$$
where $C$ is a constant and $L:(0,\infty)\to (0,\infty)$ is a function that slowly varies at infinity, i.e. $L(ct)/L(t)\stackrel{t\to \infty}{\to} 1$ for any constant $c>0$. Then, as $n\to \infty$,
$$V^{(n)}_.:=\frac{1}{n^H \sqrt{L(n)}} \sum_{k=1}^{\left\lfloor n.\right\rfloor} U_k \to X_. \ ,$$
where $(X_t)_{t\geq 0}$ is a NC-fractional Brownian motion of Hurst index $H$. To be more specific, for all times $t_1,\ldots,t_k \geq 0$, one has $\vp\big(V^{(n)}_{t_1} \cdots V^{(n)}_{t_k}\big) \to \vp\big( X_{t_1} \cdots X_{t_k}\big)$ as $n\to \infty$.
\end{proposition}


\smallskip

The NC-fBm also appeared more recently through the following result on a possible link between the law of the process and the asymptotic spectral behaviour of growing fractional matrices (keep in mind, however, that the correspondance between NC-fBm and large random matrices is a much less understood topic than in the free situation):

\begin{proposition}{\cite[Theorem 1]{matrix-approx}}
In a classical probability space, consider a collection 
$$\{B^{(n)}(i,j); \, n\geq 1 \, , \, 1\leq i\leq j\leq n\}$$
of independent fractional Brownian motions with common Hurst index $H>\frac12$, and define the sequence of symmetric random matrices $(M^{(n)})_{n\geq 1}$ along the formula
$$M^{(n)}_t (i,j)=\frac{B^{(n)}_t(i,j)}{\sqrt{n}} \ \text{for} \ \ 1\leq i<j\leq n \quad , \quad M^{(n)}_t (i,i)=\frac{\sqrt{2} B^{(n)}_t(i,i)}{\sqrt{n}} \ .$$
Denote by $\{\la_1^{(n)} \leq \ldots \leq \la_n^{(n)}\}_{n\geq 1}$ the corresponding random sequence of eigenvalues, and set $\mu^n_t :=\frac{1}{n} \sum_{i=0}^n \delta_{\la_i^{(n)}}$. Then, for every continuous function $f:\R \to \R$ and every $t\geq 0$, one has a.s.
$$\int_{\R} f(x) \, \mu^{(n)}_t(dx) \stackrel{n\to \infty}{\longrightarrow} \int_{\R} f(x) \, \mu_t(dx) \ ,$$
where $\mu_t$ stands for the semicircular distribution of mean $0$ and variance $t^{2H}$ (see Lemma \ref{lem:law-semicircular}).  
\end{proposition}

\

In the present study, we propose to go ahead with the analysis of the properties of the NC-fBm by adressing another natural question, namely: how to integrate with respect to this process?

\smallskip

Let us slightly specialize the problem by taking account of the algebra setting. Given a NC-fBm $\{X_t\}_{t\in [0,1]}$ in a NC-probability space $(\ca,\vp)$ (with a given Hurst index $H\in (0,1)$) and two paths $Y,Z:[0,1] \to \ca$ (in a class to be determined), we are looking for a natural interpretation of the integral 
\begin{equation}\label{integral-too-general}
\int_s^t Y_u dX_u Z_u \quad  , \quad s,t\geq 0 \ ,
\end{equation}
that would (for instance) extend the existing constructions in the specific free case $H=\frac12$.

\smallskip

In order to achieve this goal (at least to some extent), and in the continuation of the analysis developed in \cite{deya-schott,deya-schott-2}, our strategy will rely on an adaptation of the so-called \enquote{pathwise} methods which have been successfully used in the classical setting to handle integration with respect to the standard fractional Brownian (\cite{coutin-qian,friz-victoir-gauss}). We will thus see how to combine this approach (whether \textit{Young integration} or \textit{rough paths theory}) with the specific topological features of the algebra setting under consideration.

\smallskip

Let us recall, at a very general level, that the \enquote{pathwise} methods are based on a subtle analysis of the local dynamics of the paths under consideration. In particular, the construction of the integral depends in an essential way on the local Hölder regularity of these paths, which, just as in the classical setting, will here be governed by the value of the Hurst index $H$ (see Lemma \ref{lem:holder-regu} below). In brief, the smaller $H$, the rougher the process and the more sophisticated the integration procedure.

\

\textit{Our results regarding the construction of the integral and its properties will be gathered in Section \ref{sec:main-results}. The rest of the paper, i.e. Section \ref{sec:proof-main-result-1} and Section \ref{sec:proof-main-result-2}, is then devoted to the details of the proofs of our two main technical propositions (Proposition \ref{main-proposition} and Proposition \ref{prop:non-convergence}, respectively), that we have postponed for the sake of clarity.}

\

As a conclusion to this introduction, we would like to emphasize the particular position of our study, at the crossroads of two theories (with a priori distinct related \enquote{communities}): non-commutative probability theory and rough paths theory. In this context, and even if our analysis deeply leans on the combination of the two theories, let us point out a few specific aspects of our results that might be of special interest to each \enquote{audience}: 

\smallskip

\noindent
$(i)$ \textit{From a rough-path-expert's perspective}. In the course of the analysis, and more precisely when $H< \frac12$, we will be led to involve a fundamental second-order path into the procedure, which will play the role of a \enquote{L{\'e}vy area} in this setting (such a consideration should not be a surprise to any rough-path user). This object morally corresponds (at least in a simplified version, see Section \ref{subsec:first-rough-case}) to the product iterated integral
\begin{equation}\label{simple-product-levy}
\int_s^t \{X_u-X_s\} dX_u \ ,
\end{equation}
provided we can give a sense to this integral. In the commutative setting, i.e. when $\ca=\R$ and $X=x$ is a one-dimensional fractional Brownian, the interpretation can be immediately derived from the (formal) integration formula
\begin{equation}\label{ipp-commut}
\int_s^t \{x_u-x_s\} dx_u=\frac12 \{x_t-x_s\}^2 \ ,
\end{equation}
which indeed yields a suitable definition for this object, for any $H\in (0,1)$. In the general non-commutative setting, the corresponding integration formula reads (still formally) as
\begin{equation}\label{ipp-non-commut}
\int_s^t \{X_u-X_s\} dX_u=\{X_t-X_s\}^2 - \int_s^t dX_u\{X_u-X_s\} \ ,
\end{equation}
but there is no reason anymore for the two integrals in (\ref{ipp-non-commut}) to be equal, which of course scuttles the simplification procedure.
The situation here is somehow analog to the case of the non-diagonal entries of the classical L{\'e}vy-area matrix above a standard multidimensional fBm, and in fact, we will observe a similar breaking phenomenon when letting the Hurst index $H$ decrease from $1$ to $0$: when $H>\frac14$, we can indeed define (\ref{simple-product-levy}) through a natural approximation procedure (Proposition \ref{main-proposition}), while for $H\leq \frac14$, the very same approximation fails to converge (Proposition \ref{prop:non-convergence}). Note however that this change of regime and this similarity with the standard multidimensional fBm are not behaviours we could have readily expected, because the two objects (the product integral (\ref{simple-product-levy}) and the non-diagonal entries of the classical fractional L{\'e}vy area) are not exactly of the same nature.

\smallskip

\noindent
$(ii)$ \textit{From a non-commutative-expert's perspective}. To the best of our knowledge, stochastic integration in the non-commutative setting is so far limited to the sole free Brownian case (or its $q$-extension, see Section \ref{subsec:q-fbm} below), where it can be seen as the direct counterpart of It{\^o}'s construction. Thus, even if essentially restricted to polynomial integrands, our construction offers a new and clearly non-trivial example of an integral driven by an irregular non-commutative process. In particular, the pathwise methods will allow us to go beyond the \enquote{free independence} condition, just as they allow to go beyond the martingale framework in the classical setting. Besides, as we will see it in Remark \ref{rk:commutator}, the study of the (simplified) product L{\'e}vy area (\ref{simple-product-levy}) happens to be closely related to the behaviour of the commutator $\big[X_s,X_t\big]:=X_sX_t-X_tX_s$, as $s,t$ are getting close to each other (a property which, to some extent, can be guessed from (\ref{ipp-non-commut})). Accordingly, through the integration issue, we will also be led to test the \enquote{local commutation default} of the process and offer another interesting interpretation of the above-mentionned change of regime: in brief, when $H\leq \frac14$, the NC-fBm becomes \enquote{too non-commutative} for the sum $\sum_{(t_i)} \big[X_{t_i},X_{t_{i+1}}\big]$ of local commutators to converge in $\ca$ (as the mesh of the subdivision $(t_i)$ tends to $0$).

\

\textit{As far as the presentation of our results is concerned, please note our intention to make the subsequent analysis easily accessible to both the rough-path and the non-commutative \enquote{communities}. For this reason, we have endeavored to make the study as self-contained as possible.}

\

\section{Integration with respect to the NC-fractional Brownian motion}\label{sec:main-results}

This section accounts for our main results, along the following organization. First, we will specify our aims and expectations regarding the construction of the integral, and recall some basic technical tools from pathwise integration theory. Then we will turn to the definition of the integral, depending on the Hurst index $H$ of the process: Young-type construction when $H>\frac12$, It{\^o}-type (or Stratonovich-type) construction when $H=\frac12$, rough-path-type construction when $H<\frac12$. Finally, we will point out some possible extensions and some limits of our approach.

\smallskip

\textit{From now on and for the rest of the section, we fix a NC-probability space $(\ca,\vp)$ and consider a NC-fBm $\{X_t\}_{t\geq 0}$ of Hurst index $H\in (0,1)$ on this space. Besides, for more simplicity, we will restrict the subsequent considerations and constructions to the time interval $[0,1]$, but the results could be readily extended to any interval $[0,T]$, $T>0$.}

\subsection{Objectives}\label{subsec:objectives}
It is a well-known and natural fact that the difficulty in constructing an integral is often correlated with the \enquote{roughness} of its driving path. In the case of the NC-fBm, and just as in the case of the standard fBm, we can easily quantify this (ir)regularity along the classical Hölder scale:

\begin{lemma}\label{lem:holder-regu}
For all $0\leq s\leq t \leq 1$, it holds that
\begin{equation}\label{holder-regu}
\|X_t-X_s\| = 2 \, |t-s|^H \ .
\end{equation}
\end{lemma}

\begin{proof}
This is an elementary argument, but we provide it for the non-initiated reader as an illustration of the specific topological property (\ref{trace-norm-2}). Observe indeed that, combining (\ref{form-wick}) and (\ref{cova-NC-fBm}), one has immediately, for all $0\leq s\leq t \leq 1$ and $r\geq 1$,
$$\vp\big( (X_t-X_s)^{2r} \big)=|NC_2(2r)|\, \vp\big( (X_t-X_s)^2 \big)^r= |NC_2(2r)|\, |t-s|^{2Hr} \ .$$
Now recall that the cardinal $|NC_2(2r)|$ of the set of the non-crossing pairings of $\{1,\ldots,2r\}$ is given by the $r$-th Catalan number, whose asymptotic behaviour is well-known and yields $|NC_2(2r)|^{1/(2r)} \to 2$ as $r\to \infty$. We are therefore in a position to apply (\ref{trace-norm-2}) and derive (\ref{holder-regu}).
\end{proof}

Going back to the above discussion, and with property (\ref{holder-regu}) in hand, we can thus expect the analysis to obey the following general principles: the smaller $H$, the more irregular the process and the more difficult the construction. In fact, as a lesson from the pathwise approach in the classical probability setting, we can expect the most serious difficulties (and accordingly the most interesting problems) to arise as soon as $H\leq \frac12$.

In order to be able to go below this fundamental $\frac12$ threshold, we will restrict our attention to a relatively simple class of integrands $Y,Z$ in (\ref{integral-too-general}), namely polynomial functions of $X$, which obviously makes sense in this algebra setting. Let us therefore rephrase our objective as follows: given two polynomials $P,Q$, how to \enquote{naturally} and \enquote{efficiently} define the integral
\begin{equation}\label{integral-general}
\int_s^t P(X_u) \mathrm{d}X_u Q(X_u) \quad , \quad 0\leq s\leq t\leq 1 \ \quad ?
\end{equation}

Note that, even if restricted to polynomial integrands, this question remains far from trivial. Consider for instance the case of the elementary integral $\int_s^t X_u \, dX_u$. Following the standard Lebesgue (or Stieltjes) procedure, we could be tempted to define this object as the limit of the Riemann sums $\sum_{(t_i)\in \Delta_{s,t}} X_{t_i} \{X_{t_{i+1}}-X_{t_i}\}$, for any subdivision $\Delta_{s,t}$ of $[s,t]$ whose mesh tends to $0$. And yet, taking e.g. the basic sequence $t_i^n:=\frac{i}{n}$, it can be checked from (\ref{cova-NC-fBm}) that, just as in the commutative case,
\begin{equation}\label{div-riem-sums}
\vp\Big( \sum_{i=0}^{n-1} X_{t_i} \{X_{t_{i+1}}-X_{t_i}\}\Big)=\frac12 (1-n^{1-2H})\ ,
\end{equation}
which tends to infinity as soon as $H<\frac12$, ruling out the standard Stieltjes procedure as a general way to define the integral in (\ref{integral-general}).

\smallskip

Before we turn to the presentation of our results, let us slightly elaborate on the few specifications we shall keep in mind regarding the desired integral, for a both natural and efficient definition:

\smallskip

\noindent
$(a)$ We would like this interpretation to be relatively \enquote{intrinsic}, that is to depend on $X$ only, and not on some approximation of the process or some particular sequence of subdivisions of the time interval.  

\smallskip

\noindent
$(b)$ We expect the resulting integral to obey natural differentiation rules, such that for instance
$$X^2_t-X^2_s=\int_s^t X_u \, dX_u +\int_s^t dX_u \, X_u  \ ,$$
or analog It{\^o}-type formulas.

\smallskip

\noindent
$(c)$ As far as possible, we would like the construction to appear as the limit of the standard (Lebesgue) construction, and the integral in (\ref{integral-general}) to appear as the limit of the standard (Lebesgue) integral 
$$\int_s^t P(X^{(n)}_u) \mathrm{d}X^{(n)}_u Q(X^{(n)}_u) \ ,$$
where $X^{(n)}$ is a sequence of smooth paths that converges to $X$ as $n\to \infty$. For a clear expression  of this robustness (or Wong-Zakaï-type) property, we will refer in the sequel to the \enquote{canonical} sequence derived from the linear interpolation of $X$ along the dyadic subdivision of $[0,1]$. Thus, for the rest of the section, we set $t_i^n:=\frac{i}{2^n}$ ($i=0,\ldots,2^n$) and denote by $\{X^{(n)}_t\}_{n\geq 0,t\in [0,T]}$ the sequence defined as
\begin{equation}\label{approx-x}
X^{(0)}_t=tX_1 \quad , \quad X^{(n)}_t:=X_{t^n_i}+2^n (t-t^n_i) \{ X_{t^n_{i+1}}-X_{t^n_i }\} \quad \text{for} \ n\geq 1 \ \text{and} \ t\in [t^n_i,t^n_{i+1}]\ .
\end{equation}
Observe that the convergence of $X^{(n)}$ to $X$ is a straightforward consequence of the $H$-Hölder regularity of $X$. Using (\ref{holder-regu}), we get more precisely: 
\begin{lemma}\label{lem:fir-or}
For all $n\geq 0$, $\varepsilon\in (0,H)$ and $0\leq s\leq t\leq 1$, it holds that
\begin{equation}\label{fir-or-pr}
\| X^{(n)}_{t}-X^{(n)}_{s}\|\leq 6  |t-s|^{H}  \quad \text{and} \quad \| (X^{(n)}-X)_t-(X^{(n)}-X)_s\|\leq 8  |t-s|^{H-\varepsilon} 2^{-n\varepsilon}\ .
\end{equation}
\end{lemma}

\

Now, as a preliminary step of our construction strategy, we need to remind the reader with a few elementary results from pathwise integration theory, as developed in \cite{gubi}, and that we directly specialize to the algebra $\ca$ under consideration.

\subsection{Technical tools from pathwise integration theory}
For $k\in \{1,2,3\}$, we set $\cs_k:=\{(t_1,\ldots,t_k) \in [0,1]^k: \ t_1 \leq \ldots \leq t_k\}$ and denote by $\cac_k([0,1];\ca)$ the set of continuous maps $g: \cs_k \to \ca$ vanishing on diagonals (i.e., $g_{t_1 \ldots t_k}=0$ when two times $t_i,t_j$ with $i\neq j$ are equal).

\smallskip

Then we define the increment operator $\delta$ along the formulas: for $g:[0,1] \to \ca$, $(\der g)_{st}:=g_t-g_s$ ($0\leq s\leq t\leq 1$), while for $h:\cs_2 \to \ca$, $(\der h)_{sut}:=h_{st}-h_{su}-h_{ut}$ ($0\leq s\leq u\leq t\leq 1$). 

\smallskip

The two basic results at the core of pathwise integration theory can now be stated as follows:

\begin{lemma}\label{lem:coho}
Let $h: \cs_2 \to \ca$ be a map such that for all $0\leq s\leq u\leq t\leq 1$, $(\delta h)_{sut}=0$. Then there exists a path $g:[0,1]\to \ca$ such that $h=\delta g$.
\end{lemma}

\begin{lemma}[Sewing lemma \cite{gubi}]\label{lemma-lambda}
Let $h:\cs_3 \to \ca$ be a map in $\text{Im} \,\, \der$ (i.e. $h=\der g$ for $g:\cs_2 \to \ca$) such that for all $0\leq s\leq u\leq t\leq 1$, 
$$\|h_{sut}\|\leq C_h\, |t-s|^{\mu} \ ,$$
for some constant $C_h>0$ and some parameter $\mu>1$. Then there exists a unique element $\Lambda h\in \cac_2^{\mu}([0,1];\ca)$ such that $\delta(\Lambda h)=h$. Besides, for all $0\leq s\leq t\leq 1$, one has
\begin{eqnarray} \label{contraction}
\| (\Lambda h)_{st}\|\leq c_\mu C_h\, |t-s|^{\mu}  ,
\end{eqnarray}
where $c_\mu :=2+2^\mu \sum_{k=1}^\infty k^{-\mu}$.
\end{lemma}

In order to efficiently combine these two lemmas within an integration procedure, we will also need the integrands (and their expansions) to satisfy suitable estimates. In the polynomial setting we restrict to, such estimates can be easily verified, but let us label them for further reference.

\smallskip

At first order, one has trivially, for every polynomial $P$ and all $U,V,Y\in \ca$,  
\begin{equation}\label{basic-ineq-1}
\big\| P(V)-P(U)\big\| \leq c_P\, \big( 1+\|U\|+\|V\| \big)^{p-1} \| V-U\|  \ .
\end{equation}
For a convenient expression of the corresponding second-order bound, let us introduce the following additional notations, borrowed from \cite{biane-speicher}, and that we will also use in our forthcoming expansions (see for instance (\ref{formal-second-order-exp})). First, for all $U,V,Y\in \ca$, we set
\begin{equation}\label{sharp-not}
(U\otimes V)\sharp Y=Y\sharp (U\otimes V):=UYV \ .
\end{equation}
Then, given a polynomial function $P(x):=\sum_{k=0}^p a_i \, x^i$ and an element $U\in \ca$, the \emph{tensor derivative} of $P$ (at $U$) is the element of the algebraic tensor product $\ca\otimes \ca$ defined as
$$\partial P(U):=\sum_{k=1}^d a_k\sum_{i=0}^{k-1} U^i \otimes U^{k-1-i}  \ .$$
Combining these two notations, the second-order control we shall rely on in the sequel can be written as
\begin{equation}\label{basic-ineq-2}
\big\| P(V)-P(U) -\partial P(U) \sharp \{V-U\} \big\| \leq c_P \, \big( 1+\|U\|+\|V\| \big)^{p-2} \|V-U\|^2  \   .
\end{equation}
This is of course nothing but a basic application of the classical Taylor estimates (in a normed algebra setting).


\subsection{The case $H>\frac12$: Young integral}\label{subsec:young}

\

\smallskip

We can finally start off our construction strategy, by focusing first on the situation where the Hurst index $H$ of the process is strictly larger than $\frac12$. Still keeping property (\ref{holder-regu}) in mind, we are thus dealing here with a \enquote{not too rough} process. In fact, this situation could be encompassed within the general framework of the so-called \textit{Young integration theory}, that springs from the seminal paper \cite{young} and readily extends the classical Stieltjes interpretation (see also \cite{lyons-book} for a thorough account on the related results, in a general Banach space). Nevertheless, for the sake of completeness, and also as a way to set the stage for the rougher situations, we prefer to give a full treatment of the problem in the specific setting we are interested in. In addition, these few details on the \enquote{Young} situation will allow us to provide the non-initiated reader with a first example of the possibilities offered by the two Lemmas \ref{lem:coho} and \ref{lemma-lambda}.

\smallskip

Our main result here reads as follows (let us recall that $X^{(n)}$ stands for the approximation of $X$ defined through (\ref{approx-x}), and that any integral driven by $X^{(n)}$ is interpreted in the classical Lebesgue sense): 

\begin{proposition}\label{prop:int-young}
Assume that $H>\frac12$. Then, for all polynomials $P,Q$, all times $0\leq s\leq t \leq 1$ and every subdivision $\Delta_{st} = \{t_0=s<t_1 <\ldots<t_\ell=t\}$ of $[s,t]$ with mesh $|\Delta_{st}|$ tending to $0$, the Riemann sum
\begin{equation}\label{riemann}
\sum_{t_i\in \Delta_{st}}  P(X_{t_i})\der X_{t_it_{i+1}}Q(X_{t_i})
\end{equation}
converges in $\ca$ as $|\Delta_{st}| \to 0$. The limit provides us with a natural interpretation of the integral $\int_s^t P(X_u) d X_u Q(X_u)$, and is such that for all $n\geq 0$ and $\varepsilon \in [0,2H-1)$,
\begin{equation}\label{resu-conv-you}
\bigg\| \int_s^t P(X^{(n)}_u) dX^{(n)}_u Q(X^{(n)}_u)-\int_s^t P(X_u) dX_u Q(X_u)\bigg\| \leq c_{H,P,Q,\varepsilon} |t-s|^{H-\varepsilon} 2^{-n\varepsilon} \ ,
\end{equation}
for some constant $c_{H,P,Q,\varepsilon} >0$. As a result, one has, based on this interpretation, 
\begin{equation}\label{ito-young}
\der P(X)_{st}=\int_s^t \partial P(X_u) \sharp dX_u \ .
\end{equation}
\end{proposition}

\begin{proof}
Set $M_{st}:=P(X_s) \der X_{st} Q(X_s)$, and for all $0\leq s\leq u\leq t\leq 1$, expand the increment $\der M_{sut}$ as
\begin{equation}\label{expansion-der-m-young}
\der M_{sut}=-\der P(X)_{su} \der X_{ut} Q(X_s)-P(X_u) \der X_{ut} \der Q(X)_{su} \ ,
\end{equation}
so that, combining (\ref{holder-regu}) and (\ref{basic-ineq-1}) , we immediately obtain $\| \der M _{sut}\big\| \leq c\, |t-s|^{2H}$. Since $2H>1$, we are in a position to apply the sewing application $\Lambda$ (defined in Lemma \ref{lemma-lambda}) to $\der M$, and using Lemma \ref{lem:coho}, we can then guarantee the existence of a path $\cj :[0,1]\to \ca$ such that $\der \cj_{st}=M_{st}-\Lambda(\der M)_{st}$ for all $0\leq s\leq t\leq 1$. As a straigthforward consequence, one has
$$\sum_{t_i\in \Delta_{st}} M_{t_i t_{i+1}}=\delta \cj_{st}+\sum_{t_i\in \Delta_{st}} \Lambda(\der M)_{t_it_{i+1}} \ ,$$
with $\big\| \sum_{t_i\in \Delta_{st}} \Lambda(\der M)_{t_it_{i+1}} \big\| \leq c \sum_{t_i\in D_{st}} |t_{i+1}-t_i|^{2H} \to 0$ as $|\Delta_{st}| \to 0$, yielding the first convergence result, as well as the identity 
\begin{equation}\label{decompo-young}
\int_s^t P(X_u) \mathrm{d} X_u Q(X_u)=\delta \cj_{st}=M_{st}-\Lambda(\der M)_{st} \ .
\end{equation}

\smallskip

The argument toward (\ref{resu-conv-you}) then goes as follows. First, observe that the above procedure can be applied to the approximation $X^{(n)}$ as well, providing a similar decomposition for the limit of the corresponding sum $\sum_{t_i\in \Delta_{st}} M^{(n)}_{t_i t_{i+1}}$. Besides, as we are here dealing with a smooth path (for fixed $n$), this limit is known to coincide with the classical Lebesgue integral, and we thus obtain the identity
\begin{equation}\label{decompo-young-smooth}
\int_s^t P(X^{(n)}_u) \, \mathrm{d}X^{(n)}_u \, Q(X^{(n)}_u)=M^{(n)}_{st}-\Lambda \big( \der M^{(n)}\big)_{st} \quad , \quad M^{(n)}_{st}:=P(X^{(n)}_s) \der X^{(n)}_{st} Q(X^{(n)}_s) \ .
\end{equation}
In order to compare the two integrals in (\ref{decompo-young}) and (\ref{decompo-young-smooth}), it now suffices to control the two differences $M-M^{(n)}$ and $\Lambda(\der M)-\Lambda(\der M^{(n)})$. The first control is an immediate consequence of (\ref{fir-or-pr}), which gives the expected bound $\big\|M_{st}-M^{(n)}_{st}\big\| \leq c_{H,P,Q,\varepsilon} |t-s|^{H-\varepsilon} 2^{-n\varepsilon}$. The second control leans on both (\ref{fir-or-pr}) and the continuity of $\Lambda$: one has first, with expansion (\ref{expansion-der-m-young}) in mind,
$$\big\|\der M_{sut}-\der M^{(n)}_{sut}\big\|\leq c_{H,P,Q,\varepsilon} |t-s|^{2H-\varepsilon} 2^{-n\varepsilon} \ ,$$
and by applying Lemma \ref{lemma-lambda} we get that $\big\|\Lambda(\der M)_{st}-\Lambda(\der M^{(n)})_{st}\big\| \leq c_{H,P,Q,\varepsilon} |t-s|^{2H-\varepsilon} 2^{-n\varepsilon}$,
which achieves the proof of (\ref{resu-conv-you}).

\smallskip

Finally, (\ref{ito-young}) is just a consequence of (\ref{resu-conv-you}), since the latter convergence property allows us to pass to the limit in the classical differentiation rule $\der P(X^{(n)})_{st}=\int_s^t \partial P(X^{(n)}_u) \sharp \mathrm{d}X^{(n)}_u$.
\end{proof}

\

The resulting \enquote{Young} integral $\int_s^t P(X_u)dX_u Q(X_u)$ thus satisfies the three (moral) requirements $(a)$-$(b)$-$(c)$ raised in Section \ref{subsec:objectives}. In fact, it should be clear to the reader that the above procedure could be readily extended to handle the general integral $\int_s^t Y_u dX_u Z_u$, where $Y:[0,1] \to \ca$, resp. $Z:[0,1] \to \ca$, is a $\ga$-Hölder path, resp. $\ga'$-Hölder path, with $\ga,\ga'\in (0,1)$ satisfying $\ga+H>1$, $\ga'+H>1$. 

\smallskip

However, such an extension will no longer be possible in the subsequent rougher situations, and we have thus prefered to stick to a unified presentation around the same starting model.  

\

Let us now turn to the more interesting case where $H\leq \frac12$, with a first brief stop on the (very) particular case $H=\frac12$.

\subsection{The free case $H=\frac12$.}\label{subsec:free-case}

\

\smallskip

As we have already recalled it in the introduction, the NC-fBm with Hurst index $H=\frac12$ is nothing but the celebrated free Brownian motion, the behaviour of which has been extensively explored in the literature. In a somewhat analogous fashion as the standard (commutative) Brownian motion, the free Brownian motion is known to satisfy a specific independence property, the so-called free independence, at the level of its disjoint increments. Based on this fundamental feature, Biane and Speicher \cite{biane-speicher} have been able to adapt the principles of the classical stochastic integration theory in the non-commutative setting and construct an It{\^o}-type integral with respect to the free Brownian. In \cite{deya-schott}, we have brought a new light on these constructions along a rough-path-type approach (similar to the one we will develop the next section), which allows for more flexibility in the integration procedure, as well as additional approximation results.

\smallskip

The following statement offers a (partial) summary of these considerations, when applied to the integral in (\ref{integral-general}), and in the spirit of the present formulation. We therein denote by $\id \times \vp \times \id$ the linear extension (to $\ca\otimes \ca\otimes \ca$) of the operator $\big( \id \times \vp \times \id \big) \big(U\otimes V \otimes W\big):= \vp(V) UW$. Besides, let us again recall that $X^{(n)}$ is the smooth approximation introduced in (\ref{approx-x}), and that integrals driven by $X^{(n)}$ are all understood in the classical Lebesgue sense.

\begin{proposition}[\cite{biane-speicher,deya-schott}]\label{prop:int-free}
Assume that $H=\frac12$. For all polynomials $P,Q$, all times $0\leq s\leq t \leq 1$ and every subdivision $\Delta_{st} = \{t_0=s<t_1 <\ldots<t_\ell=t\}$ of $[s,t]$ with mesh $|\Delta_{st}|$ tending to $0$, the Riemann sum
\begin{equation}\label{ito-free}
\sum_{t_i\in \Delta_{st}}  P(X_{t_i})\der X_{t_it_{i+1}}Q(X_{t_i})
\end{equation}
converges in $\ca$ as $|\Delta_{st}| \to 0$. We denote the limit by $\int_s^t P(X_u) dX_u Q(X_u)$ and define the related \enquote{Stratonovich} integral by the formula
\begin{align}\label{strato-free}
&\int_s^t P(X_u) (\circ dX_u) Q(X_u)\\
&:=\int_s^t P(X_u) dX_u Q(X_u)+\frac12 \int_s^t \big( \id \times \vp \times \id\big)\big[ \partial P (X_u) \otimes Q(X_u)+P(X_u)\otimes \partial Q(X_u) \big] \ .
\end{align}
Then, as $n\to \infty$, it holds that
\begin{equation}\label{resu-conv-free}
\int_s^t P(X^{(n)}_u) dX^{(n)}_u  Q(X^{(n)}_u) \to \int_s^t P(X_u) (\circ dX_u)  Q(X_u) \ ,
\end{equation}
and one has in particular
\begin{equation*}
\der P(X)_{st}=\int_s^t \partial P(X_u) \sharp (\circ dX_u) \ .
\end{equation*}
\end{proposition}

\

To be more specific, the convergence of the sum in (\ref{ito-free}) is a consequence of \cite[Theorem 3.2.1]{biane-speicher}, while the approximation property (\ref{resu-conv-free}) follows from the combination of the results of Proposition 4.10, Proposition 4.16 and Proposition 5.5 in \cite{deya-schott}.

\smallskip

When dealing with the free Brownian motion, we thus observe a similar It{\^o}/Stratonovich duality as in classical stochastic integration theory (with respect to the Wiener process), which, to some extent, offers a natural transition between the previous Young situation $H>\frac12$ (with convergence of the \enquote{It{\^o}} integral, along (\ref{riemann})) and the forthcoming rough situation $H<\frac12$ (with convergence of the \enquote{Stratonovich} integral, along (\ref{corrected-riemann}) and Remark \ref{link-h-1-2}). Also, as can be seen from (\ref{resu-conv-free}), and just as in the standard Wiener case, the Stratonovich interpretation turns out to be the most robust one as far as approximation of the driver is concerned: it therefore provides us with the \enquote{solution} to our three requirements $(a)$-$(b)$-$(c)$ in Section \ref{subsec:objectives}.

\subsection{The first rough case: $H\in (\frac13,\frac12)$}\label{subsec:first-rough-case}

\

\smallskip

As soon as $H<\frac12$, both previous strategies clearly fail: the process is not regular enough for the Young method (based on classical Riemann sums) to work, and its disjoint increments are no longer (freely) independent, ruling out the free-case procedure.

\smallskip

In the continuation of \cite{deya-schott,deya-schott-2}, and borrowing some ideas from rough paths theory, we propose to introduce a construction based on the consideration of \textit{local second-order expansions} and \textit{corrected Riemann sums}, which will at least allow us to cover the case $H\in (\frac13,\frac12)$.

\smallskip

At this point, we need to mention that, at a theoretical level, this setting and the below developments are very close to the analysis carried out in \cite[Section 4]{deya-schott}. The latter reference indeed contains a possible approach to integration with respect to any $\ga$-Hölder NC-process with $\ga\in (\frac13,\frac12)$, which, due to (\ref{holder-regu}), is exactly the regularity condition that prevails here. Nevertheless, the results of \cite[Section 4]{deya-schott} leans on the extensive use of some a-priori-defined object called a \textit{product L{\'e}vy area}, and that is expected to satisfy very specific conditions (we can then check these conditions in the free Brownian case \cite[Section 5]{deya-schott} or in the $q$-Brownian situation \cite[Section 3]{deya-schott-2}). 

\smallskip

Unfortunately, owing to the strong dependency of the increments of the NC-fBm, we have not been able to exhibit such a product L{\'e}vy area above the process (we suspect that such an object does not even exist in this case, at least not \textit{stricto sensu}, see Section \ref{subsec:beyond-poly} below). Instead, we will rely on some \enquote{weaker product L{\'e}vy area}, which does not meet all the requirements of \cite[Definition 4.4]{deya-schott}, but which will be sufficient for our purpose. The construction of this object is the topic of Proposition \ref{main-proposition} below, and the main technical result of our analysis.

\smallskip

As an introduction to this central property, let us briefly recall that in the rough-path procedure, the consideration of corrected Riemann sums derives from the formal second-order expansion of the integral at stake, just as the consideration of classical Riemann sums morally stems from a first-order expansion. Thus, in the situation we are interested in, we consider that, at first order,
\begin{equation}\label{approx-first-order}
\int_s^t P(X_u) dX_u Q(X_u) \approx P(X_s) (\der X)_{st} Q(X_s) \ ,
\end{equation}
yielding the main term in (\ref{riemann}), while at second order, we have, owing to (\ref{basic-ineq-2}),
\small
\begin{align}
&\int_s^t P(X_u) dX_u Q(X_u) \nonumber\\
&\approx P(X_s) (\der X)_{st} Q(X_s)+\int_s^t \der P(X)_{su} dX_u Q(X_s)+\int_s^t P(X_s) dX_u \der Q(X)_{su}\nonumber\\
&\approx P(X_s) (\der X)_{st} Q(X_s)+\Big( \int_s^t (\partial P(X_s) \sharp \der X_{su})dX_u\Big) Q(X_s)+P(X_s) \Big(\int_s^t  dX_u (\partial Q(X_s) \sharp \der X_{su})\Big)\ ,\label{formal-second-order-exp}
\end{align}
\normalsize
which will ultimately lead us to the desired local correction. The rigorous implementation of this idea, that is the treatment of the implicit remainder in expansion (\ref{formal-second-order-exp}), will then be made possible through the combination of the technical Lemmas \ref{lem:coho} and \ref{lemma-lambda}, along a similar pattern as in the proof of Proposition \ref{prop:int-young}. However, for this machinery to work, we still need to \enquote{feed} it with a proper definition of the second-order objects involved in (\ref{formal-second-order-exp}). In other words, we still need to give an a-priori sense to (or to \enquote{explicitly construct}) the two integrals 
\begin{equation}\label{formal-iterated-integrals}
\int_s^t (\partial P(X_s) \sharp \der X_{su})dX_u \quad \text{and} \quad \int_s^t  dX_u (\partial Q(X_s) \sharp \der X_{su}) \ .
\end{equation}
In these expressions, observe that neither $\partial P(X_s)$ nor $\partial Q(X_s)$ depend on the integration variable $u$, so that the integrals can overall be regarded as the \textit{product iterated integrals of $X$} or the \textit{product L{\'e}vy areas above $X$} (\enquote{applied} to $\partial P(X_s)$ and $\partial Q(X_s)$). Observe also (still at a heuristic level for the moment) that the second integral in (\ref{formal-iterated-integrals}) can be easily recovered from the first one: indeed, as $X_t$ is a self-adjoint element in $\ca$ (for every $t$), one has morally, for every $U\in \ca$,
$$\Big(\int_s^t  dX_u U \der X_{su}\Big)^\ast=\int_s^t \der X_{su} U^\ast dX_u \ .$$

With these few ideas in mind, let us turn to the actual construction procedure, which, as in \cite{coutin-qian}, will lean on an approximation of these objects. Namely, consider the approximation $(X^{(n)})_{n\geq 0}$ of $X$ given by (\ref{approx-x}) and define the sequence of approximated L{\'e}vy areas by the natural formula: for all $n\geq 0$ and $U\in \ca$,
\begin{equation}\label{levy-area-appr}
\mathbb{X}^{2,(n)}_{st}[U]:=\int_s^t \der X^{(n)}_{su} U \mathrm{d}X^{(n)}_u \ ,\quad 0\leq s\leq t\leq 1\ ,
\end{equation}
where the integral is here understood in the classical Lebesgue sense. Our objective now is to show the convergence of this sequence as $n\to\infty$, that is as $X^{(n)}$ converges to $X$. As it can be guessed from the proof of Proposition \ref{prop:int-young}, the pathwise method also requires us to exhibit suitable controls on the limit, regarding whether the time variables $s,t$ or the \enquote{fixed} integrand $U$ in (\ref{levy-area-appr}). 

\smallskip 

\begin{remark}\label{rk:commutator}
Let us briefly go back here to the discussion we have launched at the end of Section \ref{sec:introduction}, and insist on the specificity of this object, that is $\int_s^t \der X^{(n)}_{su} U \mathrm{d}X^{(n)}_u$, with respect to its classical commutative counterpart. Note indeed that if $\ca$ were a commutative algebra, or more generally if the variables $\{U,X_t; \, t\geq 0\}$ all commuted, then expression (\ref{levy-area-appr}) would of course reduce to $\frac12 U (\der X^{(n)}_{st})^2$, providing an immediate answer to the above convergence issue, for any $H\in (0,1)$ (it is a well-known fact that the rough-path approach is not relevant when applied to a one-dimensional - and so, commutative - Hölder process).
In a general algebra, this question is no longer trivial and is in fact closely related to the local \enquote{non-commutativity} of the process under consideration. For instance, it is easy to see that 
\begin{equation}\label{levy-area-commutation}
\mathbb{X}^{2,(n)}_{01}[1]=\frac12 X_1^2+\frac12\sum_{i=0}^{2^n-1}\big[X_{t_i^n},X_{t_{i+1}^n} \big] \quad , \quad \text{with} \  \big[X_{t_i^n},X_{t_{i+1}^n} \big]:= X_{t_i^n}X_{t_{i+1}^n}-X_{t_{i+1}^n}X_{t_{i}^n} \ ,
\end{equation}
and in light of this expression, the question (morally) is therefore to know whether the sum $\sum_i [X_{t_i^n},X_{t_{i+1}^n}]$ of \enquote{infinitesimal commutators} converges as $n\to \infty$. At a heuristic level, the problem can thus be interpreted as follows: the more \enquote{locally commutative} the process (i.e., the smaller $[X_{t_i^n},X_{t_{i+1}^n}]$), the more chances the sum, and accordingly the sequence $\mathbb{X}^{2,(n)}_{01}[1]$ of (simplified) product L{\'e}vy areas, to converge.
\end{remark}

\smallskip

In order to prove the convergence of $(\mathbb{X}^{2,(n)})_{n\geq 0}$ in the present situation, that is when $H\in (\frac13,\frac12$), we actually need to reduce the class of possible \enquote{fixed} integrands $U$ in (\ref{levy-area-appr}), at least in a way that still encompasses our target integrals in (\ref{formal-iterated-integrals}). To this end, we introduce, for all $t\in [0,1]$, the unital subalgebra $\ca_t$ generated by $(X_s)_{0\leq s\leq t}$, i.e.
$$\ca_t:=\big\{ \la_0\, 1+\sum_{i=1}^n \la_i X_{t_1^i}X_{t_2^i} \cdots X_{t_{p_i}^i} \ : \ n\geq 1 \ , \ \la_i\in \R \ , \ p_i\geq 1 \ , \ 0\leq t_j^i \leq t\big\} \ .$$
The desired property can now be stated as follows:

\begin{proposition}\label{main-proposition}
Assume that $H\in (\frac14,\frac12)$. Then, for all $0\leq s\leq t\leq 1$ and $U\in \ca_s$, the sequence $\mathbb{X}_{st}^{2,(N)}[U]$ converges in $\ca$ as $N\to \infty$, and the limit, that we denote by $\mathbb{X}_{st}^{2}[U]$, satisfies the following properties:

\smallskip

\noindent
$(i)$ For all $0\leq s\leq t\leq 1$, $\mathbb{X}^2_{st}\in \cl(\ca_s,\ca)$.

\smallskip

\noindent
$(ii)$ For all $0\leq s\leq u\leq t\leq 1$ and $U\in \ca_s$, 
\begin{equation}\label{chen}
\mathbb{X}^2_{st}[U]-\mathbb{X}^2_{su}[U]-\mathbb{X}^2_{ut}[U]=\der X_{su} U \der X_{ut}\ .
\end{equation}

\smallskip

\noindent
$(iii)$ For all $\varepsilon\in (0,2H-\frac12)$, $\varepsilon'\in [0,H)$, there exist constants $c_{H,\varepsilon},c_{H,\varepsilon,\varepsilon'} >0$ such that for all $0\leq s\leq t\leq 1$, $N\geq 0$, $m\geq 0$, $N\leq N_1,\ldots,N_m\leq \infty$, $1\leq \iota\leq m$ and $0\leq u_j\leq v_j\leq s$ ($j=1,\ldots,m$),
\begin{equation}\label{roughness-general-1}
\big\| \big\{ \mathbb{X}^2_{st}-\mathbb{X}^{2,(N)}_{st}\big\}\big[ \der X_{u_1 v_1} \cdots \der X_{u_m v_m}\big] \big\|
 \leq c_{H,\varepsilon} 2^m \frac{|t-s|^{2H-\varepsilon}}{2^{N\varepsilon}} \prod_{j=1,\ldots,m} |u_j-v_j|^H\ ,
\end{equation}
and
\begin{align}
&\big\| \mathbb{X}^{2,(N)}_{st}\big[ \der X^{(N_1)}_{u_1 v_1} \cdots \der(X^{(N_\iota)}-X)_{u_\iota v_\iota} \cdots \der X^{(N_m)}_{u_m v_m}\big] \big\|\nonumber\\
&\hspace{1cm}\leq c_{H,\varepsilon,\varepsilon'} 2^m |t-s|^{2H-\varepsilon}\frac{|v_\iota-u_\iota|^{H-\varepsilon'}}{2^{N_\iota\varepsilon'}} \prod_{\substack{j=1,\ldots,m\\ j\neq \iota}} |u_j-v_j|^H\ ,\label{roughness-general-2}
\end{align}
where we have used the convention $X^{(\infty)}:=X$.
\end{proposition}

For the sake of clarity, we have postponed the (technical) proof of this result to Section \ref{sec:proof-main-result-1}. Let us now see how we can lean on this object $\mathbb{X}^2$ (and the related approximation results) to offer a satisfying interpretation of the general integral in (\ref{integral-general}). For a convenient statement of the result, we set, along a similar convention as in (\ref{sharp-not}), and for all $0\leq s \leq t\leq 1$, $U,V \in \ca_s$, 
$$(U\otimes V)\sharp \mathbb{X}^2_{st}:=U\mathbb{X}^2_{st}[V]\quad , \quad \mathbb{X}^{2,\ast}_{st}\sharp (U\otimes V):=\mathbb{X}^2_{st}[U^\ast]^\ast V\ ,$$
and then linearly extend these two notations to any element of the algebraic tensor product $\ca_s \otimes \ca_s$.

\smallskip

\begin{proposition}\label{prop:int-gen}
Fix $H\in (\frac13,\frac12)$, and let $P,Q$ be two polynomials. For all $0\leq s\leq t \leq 1$ and every subdivision $\Delta_{st} = \{t_0=s<t_1 <\ldots<t_\ell=t\}$ of $[s,t]$ with mesh $|\Delta_{st}|$ tending to $0$, the corrected Riemann sum
\begin{equation}\label{corrected-riemann}
\sum_{t_i\in \Delta_{st}} \Big\{ P(X_{t_i})(\der X)_{t_it_{i+1}}Q(X_{t_i})+(\partial P (X_{t_i}) \sharp \mathbb{X}^2_{t_it_{i+1}}) Q(X_{t_i})+P(X_{t_i})(\mathbb{X}^{2,\ast}_{t_it_{i+1}} \sharp \partial Q(X_{t_i}))\Big\}
\end{equation}
converges in $\ca$ as $|\Delta_{st}| \to 0$. The limit, that we denote by $\int_s^t P(X_u)d X_u Q(X_u)$, is such that for all $n\geq 0$ and $\varepsilon \in (0,\frac12 (3H-1))$,
\begin{equation}\label{resu-conv-rough}
\bigg\| \int_s^t P(X^{(n)}_u) dX^{(n)}_u  Q(X^{(n)}_u)-\int_s^t P(X_u) dX_u  Q(X_u)\bigg\| \leq c_{H,P,Q,\varepsilon} |t-s|^{H-\varepsilon} 2^{-n\varepsilon} \ ,
\end{equation}
for some constant $c_{H,P,Q,\varepsilon} >0$, and so, based on this construction, one has
\begin{equation}\label{ito-rough}
\der P(X)_{st}=\int_s^t \partial P(X_u) \sharp dX_u \ .
\end{equation}
\end{proposition}

\

This interpretation of the integral thus clearly meets requirements $(b)$-$(c)$ of Section \ref{subsec:objectives}. Regarding condition $(a)$, we can only assert that, as far the driver $X$ is concerned, the above definition is as intrinsic as possible (the usual and \enquote{more intrinsic} Riemann sums could indeed diverge, as shown in (\ref{div-riem-sums})), but it still involves an a-priori-defined object $\mathbb{X}^2$ whose construction may depend on the chosen approximation $X^{(n)}$ of $X$. This is a standard phenomenon in rough paths theory.

\smallskip

Thanks to the results of Lemma \ref{lem:fir-or} and Proposition \ref{main-proposition}, Proposition \ref{prop:int-gen} could essentially be derived from the considerations of \cite[Section 4]{deya-schott} (applied to the particular integral in (\ref{integral-general})). However, as we evoked it earlier, the properties exhibited in Proposition \ref{main-proposition} (and especially the two estimates (\ref{roughness-general-1}) and (\ref{roughness-general-2})) are not exactly the same as those appearing in the central definition \cite[Definition 4.4]{deya-schott}. Therefore, for both clarity and rigor, we prefer to review the main arguments behind the transition from Proposition \ref{main-proposition} to Proposition \ref{prop:int-gen}. This will also allow us to emphasize the similarities with the Young procedure of Section \ref{subsec:young}.

\begin{proof}[Proof of Proposition \ref{prop:int-gen}]
We follow the pattern of the proof of Proposition \ref{prop:int-young}, starting this time from the path
\begin{equation}\label{defi-m-rough}
M_{st}:=P(X_s)\der X_{st}Q(X_s)+(\partial P (X_s) \sharp \mathbb{X}^2_{st}) Q(X_s)+P(X_s)(\mathbb{X}^{2,\ast}_{st} \sharp \partial Q(X_s)) \ .
\end{equation}
For all $0\leq s\leq u\leq t\leq 1$, the increments of $\der M_{sut}$ can be readily expanded as 
\begin{align}
\der M_{sut}=&\big[-\der P(X)_{su} \der X_{ut} Q(X_s)+\big( \partial P (X_s) \sharp \der \mathbb{X}^2_{sut} \big) Q(X_s)\big]\nonumber\\
& +\big[-P(X_u) \der X_{ut} \der Q(X)_{su}+P(X_s) \big( \der \mathbb{X}^{2,\ast}_{sut} \sharp \partial Q(X_s) \big) \big]\nonumber\\
&\big[ \big( \partial P (X_s) \sharp \mathbb{X}^2_{ut} \big) Q(X_s)-\big( \partial P (X_u) \sharp \mathbb{X}^2_{ut} \big) Q(X_u)\big]\nonumber\\
&+\big[ P(X_s) \big( \mathbb{X}^{2,\ast}_{ut} \sharp \partial Q(X_s) \big)-P(X_u) \big(  \mathbb{X}^{2,\ast}_{ut} \sharp \partial Q(X_u) \big)\big] \ .\label{expansion-delta-m}
\end{align}
Let us now estimate each term into bracket separetely. As far as the third term is concerned, one has naturally
\begin{align*}
&\big( \partial P (X_s) \sharp \mathbb{X}^2_{ut} \big) Q(X_s)-\big( \partial P (X_u) \sharp \mathbb{X}^2_{ut} \big) Q(X_u)\\
&\hspace{1cm}=-\big( \der(\partial P (X))_{su} \sharp \mathbb{X}^2_{ut} \big) Q(X_s)-\big( \partial P (X_u) \sharp \mathbb{X}^2_{ut} \big) \der Q(X)_{su} \ ,
\end{align*}
which, using estimate (\ref{roughness-general-1}) (with $N=0$), entails that
$$\big\| \big( \partial P (X_s) \sharp \mathbb{X}^2_{ut} \big) Q(X_s)-\big( \partial P (X_u) \sharp \mathbb{X}^2_{ut} \big) Q(X_u) \big\| \leq c_{H,P,Q,\varepsilon}  |t-s|^{3H-\varepsilon} \ ,$$
for any $\varepsilon \in (0,2H-\frac12)$. The same strategy and estimate apply of course to the fourth term in (\ref{expansion-delta-m}).

\smallskip

As for the first two terms, we can use identity (\ref{chen}) to write
$$\big( \partial P (X_s) \sharp \der \mathbb{X}^2_{sut} \big) Q(X_s)=\big( \partial P (X_s) \sharp \der X_{su} \big) \der X_{ut} Q(X_s) \ ,$$
and so
\begin{align*}
&\big\| -\der P(X)_{su} \der X_{ut} Q(X_s)+\big( \partial P (X_s) \sharp \der \mathbb{X}^2_{sut} \big) Q(X_s)\big\|\\
&\hspace{1cm}=\big\| \big\{\der P(X)_{su}- \partial P (X_s) \sharp \der X_{su}\big\} \der X_{ut} Q(X_s)\big\| \leq c_{H,P,Q}  |t-s|^{3H} \ ,
\end{align*}
where we have combined (\ref{holder-regu}) and (\ref{basic-ineq-2}) to get the last inequality. Besides, using again (\ref{chen}), it is easy to check that
$$P(X_s) \big( \der \mathbb{X}^{2,\ast}_{sut} \sharp \partial Q(X_s) \big)=P(X_s)\der X_{ut} \big( \partial Q(X_s) \sharp \der X_{su}\big) \ ,$$
which, along the same argument as above, entails that
$$\big\| -P(X_u) \der X_{ut} \der Q(X)_{su}+P(X_s) \big( \der \mathbb{X}^{2,\ast}_{sut} \sharp \partial Q(X_s) \big)\big\| \leq c_{H,P,Q}  |t-s|^{3H} \ .$$
Going back to (\ref{expansion-delta-m}), we have thus shown that for all $0\leq s\leq u\leq t\leq 1$ and $\varepsilon \in (0,2H-\frac12)$, $\big\| \der M_{sut}\big\|\lesssim |t-s|^{3H-\varepsilon}$. Since $H>\frac13$, we are here in the very same position as in the proof of Proposition \ref{prop:int-young} (at least when picking $\varepsilon \in (0,3H-1)$), and following the same arguments (that is, combining Lemmas \ref{lem:coho} and \ref{lemma-lambda}), we can conclude about the existence of a path $\cj:[0,1]\to \ca$ such that for all $0\leq s\leq t\leq 1$, one has
$$\lim_{|\Delta_{st}|\to 0} \sum_{t_i\in \Delta_{st}}M_{t_i t_{i+1}}=\der \cj_{st}=M_{st}-\Lambda(\der M)_{st}=: \int_s^t P(X_u)dX_u Q(X_u) \ ,$$
which corresponds to the first part of our assertion.

\smallskip

The estimate (\ref{resu-conv-rough}) can again be shown along the same principles as in the proof of Proposition \ref{prop:int-young}. Just as in the latter proof, we can first decompose the approximated integral (for any fixed $n$) as
$$\int_s^t P(X^{(n)}_u) \, \mathrm{d}X^{(n)}_u \, Q(X^{(n)}_u)=M^{(n)}_{st}-\Lambda \big( \der M^{(n)}\big)_{st} \ ,$$
where $M^{(n)}$ is obtained by replacing $(X,\mathbb{X}^2)$ with $(X^{(n)},\mathbb{X}^{2,(n)})$ in (\ref{defi-m-rough}). Then, in order to control the differences $M^{(n)}-M$ and $\der M^{(n)}-\der M$, we can rely on the combination of  (\ref{fir-or-pr}), (\ref{roughness-general-1}) and (\ref{roughness-general-2}). For instance, writing
\begin{align*}
&(\partial P (X^{(n)}_s) \sharp \mathbb{X}^{2,(n)}_{st}) Q(X^{(n)}_s)-(\partial P (X_s) \sharp \mathbb{X}^2_{st}) Q(X_s)\\
&=(\{\partial P (X^{(n)}_s)-\partial P (X_s)\} \sharp \mathbb{X}^{2,(n)}_{st}) Q(X^{(n)}_s)+(\partial P (X_s) \sharp \{\mathbb{X}^{2,(n)}_{st}-\mathbb{X}^2_{st}\}) Q(X^{(n)}_s)\\
&\hspace{3cm}+(\partial P (X_s) \sharp \mathbb{X}^{2}_{st})\{Q(X^{(n)}_s)- Q(X_s)\}
\end{align*}
we can easily bound the first term using (\ref{fir-or-pr}) and (\ref{roughness-general-2}), and the last two terms using (\ref{fir-or-pr}) and (\ref{roughness-general-1}), which gives here
$$\big\| (\partial P (X^{(n)}_s) \sharp \mathbb{X}^{2,(n)}_{st}) Q(X^{(n)}_s)-(\partial P (X_s) \sharp \mathbb{X}^2_{st}) Q(X_s)\big\| \leq c_{H,P,Q,\varepsilon} \frac{|t-s|^{2H-\varepsilon}}{2^{n\varepsilon}} \ ,$$
for any $\varepsilon \in (0,2H-\frac12)$. Similar considerations allow us to control $\der M^{(n)}-\der M$ (keeping expansion (\ref{expansion-delta-m}) in mind), providing finally, for all $0\leq s\leq u\leq t\leq 1$ and $\varepsilon \in (0,2H-\frac12)$,
$$\big\|M^{(n)}_{st}-M_{st}\big\| \leq c_{H,P,Q,\varepsilon} \frac{|t-s|^{2H-\varepsilon}}{2^{n\varepsilon}}  \quad , \quad \big\|\der M^{(n)}_{sut}-\der M_{sut}\big\| \leq c_{H,P,Q,\varepsilon} \frac{|t-s|^{3H-2\varepsilon}}{2^{n\varepsilon}} \ .$$
Picking $\varepsilon \in (0,\frac12 (3H-1))$, the conclusion (that is, (\ref{resu-conv-rough})) follows from the continuity properties of $\Lambda$.

\smallskip

Once endowed with (\ref{resu-conv-rough}), and just as in the proof of Proposition \ref{prop:int-young}, identity (\ref{ito-rough}) is immediately derived from the classical differentiation rule $\der P(X^{(n)})_{st}=\int_s^t \partial P(X^{(n)}_u) \sharp \mathrm{d}X^{(n)}_u$. 
\end{proof}

\smallskip

\begin{remark}\label{link-h-1-2}
In both Propositions \ref{main-proposition} and \ref{prop:int-gen}, we could also have included (without any change in the statements and their proofs) the situation where $H\in [\frac12,1)$. In fact, when doing so, the resulting interpretation of the integral happens to be consistent with the previous constructions, that is with the interpretations of Proposition \ref{prop:int-young} and Proposition \ref{prop:int-free}. When $H>\frac12$, we can rely on (\ref{roughness-general-1}) to assert that, as $|\Delta_{st}|\to 0$,
$$\sum_{t_i\in \Delta_{st}} \Big\{ (\partial P (X_{t_i}) \sharp \mathbb{X}^2_{t_it_{i+1}}) Q(X_{t_i})+P(X_{t_i})(\mathbb{X}^{2,\ast}_{t_it_{i+1}} \sharp \partial Q(X_{t_i}))\Big\} \to 0 \ ,$$
so that the limit of the sum in (\ref{corrected-riemann}) indeed reduces to the limit of the classical Riemann sums in (\ref{riemann}). When $H=\frac12$, this consistency property is a consequence of Corollary 4.13 and Proposition 5.6 in \cite{deya-schott}: according to the latter results, the limit of the corrected Riemann sums (\ref{corrected-riemann}) more specifically coincides with the Stratonovich integral defined through (\ref{strato-free}). 
\end{remark}

\subsection{Rougher situations}\label{subsec:rougher-situations}

\

\smallskip

At this point, the - theoretical! - extension of our construction procedure to smaller $H$ should certainly be clear to the reader: for $H\in (\frac14,\frac13]$ (and then $H\in (\frac15,\frac14]$, ...), we formally expand the integral in (\ref{integral-general}) at order $3$ (and then at order $4$,...) and study the existence of the successive \enquote{product iterated integrals} that arise in the development.

\smallskip

When $H\in (\frac14,\frac13]$, and even if we prefer to skip the examination of the related details (or perhaps postpone it to a future report), we are relatively confident about the success of the method, which should lead to a similar result as in Proposition \ref{prop:int-gen}, by considering of course third-order-corrected Riemann sums. Thus, on top of the \enquote{product L{\'e}vy areas} $\mathbb{X}^{2},\mathbb{X}^{2,\ast}$ constructed in Proposition \ref{main-proposition} (note indeed that the latter statement holds true for any $H\in (\frac14,\frac12)$), the strategy would here require us to investigate the existence of the \enquote{product L{\'e}vy volumes} above $X$, corresponding morally to the third-order iterated integrals
\begin{equation}\label{product-levy-volumes}
\int_{(u,v,w)\in \mathcal{D}^{(i)}_{s,t}}dX_u U dX_v V dX_w \ , 
\end{equation}
for $U,V$ fixed in $\ca_s$, and where the domains $\mathcal{D}^{(i)}_{s,t}$ ($i=1,\ldots,6$) correspond to the six ordered sets composing $[s,t]^3$ (for instance, $\mathcal{D}^{(1)}_{s,t}=\{s\leq u\leq v\leq w\leq t\}$, $\mathcal{D}^{(2)}_{s,t}=\{s\leq u\leq w\leq v\leq t\}$,...). We think that the study of the integrals in (\ref{product-levy-volumes}) can certainly be done along the arguments of the proof of Proposition \ref{main-proposition} (observe for instance that the important controls on the covariances in Lemma \ref{lem:estim-cova-1}, Lemma \ref{lem:estim-cova-2} and Corollary \ref{cor:estim-cova-2-gene} are valid for any $H>\frac14$), but of course this still needs to be checked through a careful analysis.

\smallskip

What we rather would like to point out is the fact that this construction procedure is actually doomed to failure as soon as $H\leq \frac14$, which can be directly seen at second order, that is at the level of the product L{\'e}vy area:

\begin{proposition}\label{prop:non-convergence}
In a non-commutative probability space $(\ca,\vp)$, consider a NC-fractional Brownian motion $\{X_t\}_{t\geq 0}$ of Hurst index $H\leq \frac14$, and let $(X^{(n)},\mathbb{X}^{2,(n)})_{n\geq 0}$ be defined through formulas (\ref{approx-x}) and (\ref{levy-area-appr}). Then it holds that
\begin{equation}\label{lower-bound}
\vp\big( \big(\mathbb{X}^{2,(n+1)}_{01}[1]-\mathbb{X}^{2,(n)}_{01}[1]\big) \big(\mathbb{X}^{2,(n+1)}_{01}[1]-\mathbb{X}^{2,(n)}_{01}[1] \big)^\ast \big) \geq c \, 2^{n(1-4H)} \ ,
\end{equation}
for some strictly positive constant $c$. In particular, the sequence $\mathbb{X}^{2,(n)}_{01}[1]$ does not converge in $(\ca,\|.\|)$ as $n$ tends to infinity.
\end{proposition}

\smallskip

The details of the proof of this proposition can be found in the subsequent Section \ref{sec:proof-main-result-2}.

\smallskip

Going back to the interpretation in Remark \ref{rk:commutator}, and especially to (\ref{levy-area-commutation}), we can thus consider that when  $H\leq \frac14$, the NC-fBm is too \enquote{locally non-commutative} for the sequence of approximated Lévy areas to converge, and accordingly for our rough-path approach to work.

\smallskip

From a technical point of view, and although we are dealing with a quite different object here, this non-convergence phenomenon is somehow similar to the issue one must face, in classical probability theory, when considering the non-diagonal entries of the L{\'e}vy-area matrix above a standard $2$-dimensional fractional Brownian motion, that is (morally) the integral $\int_s^t \der B^{(1)}_{su} \, dB^{(2)}_u$, where $B^{(1)},B^{(2)}$ stand for independent fractional Brownian motions of common Hurst index $H\leq \frac14$, defined on a classical probability space $(\Omega,\mathcal{F},\mathbb{P})$. It is indeed a well-known fact (see for instance \cite[Proposition 30]{coutin-qian}) that the corresponding sequence of approximated L{\'e}vy areas, derived from some \enquote{canonical} approximation $(B^{(1),n},B^{(2),n})$ of $(B^{(1)},B^{(2)})$, also fails to convergence (even in probability) as $n\to \infty$.

\smallskip

In the latter commutative setting, i.e. when working with the above integral $\int_s^t \der B^{(1)}_{su} dB^{(2)}_u$ for $H\leq \frac14$, a possible way to overcome the non-convergence issue  is to extend our interpretation of iterated integrals at a more abstract level, by considering the so-called class of non-geometric rough paths, and then use this additional flexibility to exhibit a suitable object above the process. Such a (highly abstract and sophisticated) procedure has for instance been implemented in \cite{nualart-tindel}. At this stage, we must admit that we have no idea whether such considerations could be adapted in the non-commutative framework to handle the product L{\'e}vy area $\int_s^t \der X_{su} U dX_u$.

\

\subsection{Possible extensions of these considerations}

\

\smallskip

As a conclusion to our investigations, and before we turn to the technical proofs of Proposition \ref{main-proposition} and Proposition \ref{prop:non-convergence}, let us briefly outline a few possible extensions of this approach to non-commutative integration, together with related open questions.

\subsubsection{Beyond polynomial integration}\label{subsec:beyond-poly}

A first general question is whether this strategy could be extended to a larger class of integrands $Y,Z:[0,1]\to \ca$ (instead of $P(X),Q(X)$) in (\ref{integral-general}). Recall that we have already addressed this issue in the Young case $H>\frac12$ (see the end of Section \ref{subsec:young}), while such an extension can indeed be obtained in the free case $H=\frac12$ using the It{\^o}-type approach developed by Biane and Speicher (see \cite{biane-speicher} for further details).

\smallskip

The rough situation $H<\frac12$ turns out to be more problematic in this regard. In view of the above developments, a first essential question here is to know whether the definition of the \enquote{product L{\'e}vy area} $\mathbb{X}^2_{st}\big[ U\big]$ in Proposition \ref{main-proposition} could be extended to more general $U$, that is beyond polynomial expressions of $\{X_r\}_{0\leq r\leq s}$. 

\smallskip

Based on (\ref{roughness-general-1}), a possible line of generalization involves elements $U$ of the form $U:=f(X_r)$, where $r\in [0,s]$ and $f$ is a function defined through the Fourier transform $f(x)=\int_{\R} e^{\imath \xi x}\mu_f(\mathrm{d}\xi)$ and satisfying $\int_{\R} e^{2|\xi|}\mu_f(\mathrm{d}\xi)$. Indeed, at least at some formal level, we have, for such a function $f$ and for every $\ga \in (0,H)$,
\begin{eqnarray*}
\big\| \mathbb{X}^2_{st}\big[ f(X_r)\big]\big\|&\leq &\sum_{m\geq 0} \bigg| \int_{\R} \frac{(\imath \xi)^m}{m!} \mu_f(\mathrm{d}\xi) \bigg| \big\| \mathbb{X}^2_{st}\big[ X_r^m\big]\big\|\\
&\leq&c_\ga |t-s|^{2\ga} \int_{\R} \bigg( \sum_{m\geq 0} \frac{|2\xi|^m}{m!}\bigg) \mu_f(\mathrm{d}\xi) \ \leq \ c_\ga |t-s|^{2\ga} \int_{\R} e^{2|\xi|}\mu_f(\mathrm{d}\xi) \ ,
\end{eqnarray*}
which still offers the required Hölder control, and thus opens a way toward the interpretation of the integral $\int_s^t g(X_u) dX_u h(X_u)$, for smooth enough functions $g,h$.

\smallskip

Then a natural attempt to go further would be to turn to the setting introduced in \cite[Section 4]{deya-schott}, and allowing for the consideration of the more flexible class of \textit{adapted controlled biprocesses} (along \cite[Definition 4.9]{deya-schott}). Unfortunately, as we already mentionned it twice, the estimates we have obtained in Proposition \ref{main-proposition} are not sufficient for a direct application of the results of \cite{deya-schott}. In other words, the operator $\mathbb{X}^2$ derived from our result is not as general as a genuine \textit{product Lévy area}, in the specific sense of \cite[Definition 4.4]{deya-schott}. Indeed, such a product Lévy area is expected to satisfy, for all $0\leq s\leq t\leq 1$ and $U\in \ca_s$,
$$\big\| \mathbb{X}^2_{st}\big[ U\big] \big\| \leq c_\ga |t-s|^{2\ga} \|U\| \ ,$$
for some $\ga>\frac13$, which is more general than our estimate (\ref{roughness-general-1}) (with $N=0$). Morally, we would here need the right-hand side of (\ref{roughness-general-1}) to be uniformly bounded over $m$, or rather the $2^m$ factor in (\ref{roughness-general-1}) to disappear. This property can indeed be checked in the free Brownian case $H=\frac12$, owing to the free independence of the disjoint increments. When $H<\frac12$, and in light of the expressions at stake in the proof of Proposition \ref{main-proposition} (see e.g. the quantities involved in (\ref{interm-quant-2})), we must say that we have serious doubts about the existence of such a uniform estimate. 

\

Let us now evoke some possible extensions at the level of the driving process itself.

\vspace{-0.35cm}

\subsubsection{More general semicircular processes}
In the standard probability setting, the rough-path approach is known to be applicable to a class of Gaussian processes that goes beyond the fractional Brownian motion (see e.g. \cite{friz-victoir-gauss} or \cite[Chapter 15]{friz-victoir-book}), and therefore we may wonder about the existence of such an extension in the non-commutative framework. 

\smallskip

Skimming through the proof of Proposition \ref{main-proposition}, the specific involvement of the covariance of the process (here, the fractional covariance (\ref{cova-NC-fBm})) is actually easy to locate. Namely, we only use the form of this covariance within the estimates of the subsequent Lemmas \ref{lem:estim-cova-1} and \ref{lem:estim-cova-2}, and as a consequence, the developments and results of Section \ref{subsec:first-rough-case} would remain true for any (Hölder) semicircular process satisfying these two estimates.

\smallskip

This being said, at this point, we are far from being able to exhibit a similar general (and essentially sharp) covariance criterion as in \cite[Theorem 15.33]{friz-victoir-book}.

\subsubsection{The $q$-fractional Brownian motion}\label{subsec:q-fbm}
It is a well-known fact in the non-commutative-probability literature (see e.g. \cite{q-gauss}) that the semicircular processes are part of a more general class, the so-called $q$-Gaussian processes (for fixed $q\in (-1,1)$), defined through the \enquote{$q$-Wick} formula 
\begin{equation}\label{form-q-wick}
\vp\big( X_{i_1}\cdots X_{i_r}\big)=\sum_{\pi \in \mathcal{P}_2(\{1,\ldots,r\})} q^{\text{Cr}(\pi)}\ka_\pi\big( X_{i_1},\ldots,X_{i_r}\big) \ ,
\end{equation} 
where, in comparison with (\ref{form-wick}), the sum runs this time over the set of all pairings of $\{1,\ldots,r\}$, and the notation $\text{Cr}(\pi)$ refers to the number of crossings in $\pi$ (the semicircular processes are thus nothing but the $0$-Gaussian processes).

\smallskip

Along this line of generalization, and for fixed $q\in (-1,1), H\in (0,1)$, we can then naturally define the $q$-fractional Brownian motion ($q$-fBm) of Hurst index $H$, the above NC-fBm corresponding to the $0$-fBm. In \cite{deya-schott-2}, we have already applied the rough-path strategy to the $q$-Brownian motion, i.e. the $q$-fBm of Hurst index $H=\frac12$, which, at least in the case $q\in [0,1)$, led us to better controls and approximation results than in the original It{\^o}-type approach of the situation (\cite{donati}). 

\smallskip

As regards the $q$-fBm $X=X^{(q,H)}$ of Hurst index $H\neq \frac12$, observe first that we are still dealing with a $H$-Hölder process (for any fixed $q\in (-1,1)$), since, with the argument of the proof of Lemma \ref{lem:holder-regu} in mind, we have here 
$$\vp\big( (X_t-X_s)^{2r} \big)^{1/(2r)}=|t-s|^{2H} \Big(\sum_{\pi \in \mathcal{P}_2(\{1,\ldots,2r\})} q^{\text{Cr}(\pi)}\Big)^{1/(2r)} \to \frac{2|t-s|^{2H}}{\sqrt{1-q}} \quad \text{as} \ r\to \infty \ .$$
When $H>\frac12$, this basic regularity property immediately allows us to mimic the Young procedure of Section \ref{subsec:young}. As for the (more interesting) case $H<\frac12$, we must say we are rather confident about the possibility to extend the considerations of both Section \ref{subsec:first-rough-case} and Section \ref{subsec:rougher-situations} to any $q\in (-1,1)$, with a similar \enquote{success} for $H>\frac14$ and \enquote{failure} for $H\leq \frac14$. Of course, this involves a careful adaptation of the proofs of Propositions \ref{main-proposition} and \ref{prop:non-convergence}, taking the $q$-parameter into account, and we expect both the upper bounds in (\ref{roughness-general-1})-(\ref{roughness-general-2}) and the lower bound in (\ref{lower-bound}) to depend on $q$ as well.

\section{Proof of Proposition \ref{main-proposition}}\label{sec:proof-main-result-1}

\begin{proof}[Proof of Proposition \ref{main-proposition}] Throughout the proof, we will denote by $A\lesssim B$ any bound of the form $A\leq c B$, where $c$ is a universal constant independent from the parameters under consideration.

\smallskip

Our main task will be to find a suitable estimate of the difference
$$\mathbb{X}^{2,(n+1)}_{st}[U]-\mathbb{X}^{2,(n)}_{st}[U] \ ,$$
for $0\leq n\leq N-1$, $0\leq s\leq t\leq 1$ and $U\in \ca_s$.

\smallskip

If $|t-s|\leq 2^{-n+1}$, then we can check explicitly (see the beginning of the proof of \cite[Theorem 3.1]{deya-schott-2} for details) that 
\begin{equation}\label{trivial-case}
\big\| \mathbb{X}^{2,(n+1)}_{st}[U]-\mathbb{X}^{2,(n)}_{st}[U]\big\| \lesssim \frac{|t-s|^{(2H-\varepsilon)}}{2^{n\varepsilon}} \|U\|\ .
\end{equation}
Let us assume from now on that
\begin{equation}\label{choice-s-t}
t_{k-1}^n \leq s<t_k^n <t_\ell^n \leq t<t_{\ell+1}^n \quad \text{with} \quad \ell-k\geq 1 \ .
\end{equation}
Using the first-order controls (\ref{fir-or-pr}) only, it can be shown that (see again the beginning of the proof of \cite[Theorem 3.1]{deya-schott-2} for details) 
\begin{equation}\label{reduction}
\big\| \big\{\mathbb{X}^{2,(n+1)}_{st}[U]-\mathbb{X}^{2,(n)}_{st}[U]\big\}-\big\{\mathbb{X}^{2,(n+1)}_{t_k^n t_\ell^n}[U]-\mathbb{X}^{2,(n)}_{t_k^n t_\ell^n}[U]\big\}\big\| \lesssim \frac{|t-s|^{(2H-\varepsilon)}}{2^{n\varepsilon}} \|U\|\ ,
\end{equation}
and we are thus left with the estimation of $\mathbb{X}^{2,(n+1)}_{t_k^n t_\ell^n}[U]-\mathbb{X}^{2,(n)}_{t_k^n t_\ell^n}[U]$. Setting $Y_i=Y_{i}^{(n)}:=\delta X_{t_{i}^{n+1}t_{i+1}^{n+1}}$, this difference actually reduces to 
\begin{align}
\mathbb{X}^{2,(n+1)}_{t_k^n t_\ell^n}[U]-\mathbb{X}^{2,(n)}_{t_k^n t_\ell^n}[U]&=\int_{t_{k}^n}^{t_\ell^n}  \delta X^{(n+1)}_{t_{k}^n u}U \mathrm{d}X^{(n+1)}_u-\int_{t_{k}^n}^{t_\ell^n} \delta X^{(n)}_{t_{k}^nu}U \mathrm{d}X^{(n)}_u\nonumber\\
&=\frac12 \sum_{i=k}^{\ell-1} \big[  Y_{2i} U Y_{2i+1}-  Y_{2i+1} U Y_{2i} \big] \ .\label{main-term-area}
\end{align}
Keeping in mind the desired estimates (\ref{roughness-general-1})-(\ref{roughness-general-2}), we henceforth consider $U$ of the two following possible forms:
\begin{align*}
&\textbf{Situation A:} \quad U:=U_1 \cdots U_m \ , \quad m\geq 1 \ , \quad U_j:=\der X_{u_j v_j}  \ , \ 0\leq u_j\leq v_j\leq s \ ;\\
&\textbf{Situation B:} \quad U:=U_1 \cdots U_\iota \cdots U_m \ , \quad 1\leq \iota\leq m \ , \\
&\hspace{4cm}U_j:=\der X^{(N_j)}_{u_j v_j}  \ \ \text{for} \ j \neq \iota \ , \ U_\iota:=\der(X^{(N_\iota)}-X)_{u_\iota v_\iota} \ , \ \ 0\leq u_j\leq v_j\leq s \ .
\end{align*}
For each of these two situations, our aim is thus to estimate (due to (\ref{trace-norm-2}))
\begin{align*}
&\bigg\|\sum_{i=k}^{\ell-1} \big[  Y_{2i} U Y_{2i+1}-  Y_{2i+1} UY_{2i} \big]\bigg\|\\
&=\lim_{r\to \infty}\vp\bigg( \bigg( \bigg(\sum_{i=k}^{\ell-1} \big[  Y_{2i} U Y_{2i+1}-  Y_{2i+1} UY_{2i}\big]\bigg) \bigg(\sum_{i=k}^{\ell-1} \big[  Y_{2i}U Y_{2i+1}-  Y_{2i+1}U Y_{2i}\big]\bigg)^{\!\!\! \ast} \bigg)^r \bigg)^{1/(2r)} \ . 
\end{align*}
For the sake of clarity, let us introduce the additional notations: for all $V,V_1,\ldots,V_m\in \ca$,
$$\mathbf{Y}_{1,i}[V]:=Y_{2i}VY_{2i+1} \quad , \quad \mathbf{Y}_{-1,i}[V]:=Y_{2i+1}VY_{2i}$$
and
$$\mathbb{Y}_{1,i}[V_1,\ldots,V_m]:=(Y_{2i},V_1,\ldots,V_m,Y_{2i+1}) \quad , \quad \mathbb{Y}_{-1,i}[V_1,\ldots,V_m]:=(Y_{2i+1},V_1,\ldots,V_m,Y_{2i}) \ .$$
The above $r$-th moment can then be expanded as
\begin{align*}
&(-1)^r \sum_{i_1,\ldots,i_{2r}} \vp\big( \big\{\big[ \mathbf{Y}_{1,i_1}(U)-\mathbf{Y}_{-1,i_1}(U)\big]\big[ \mathbf{Y}_{1,i_2}(U^\ast)-\mathbf{Y}_{-1,i_2}(U^\ast)\big]\big\} \cdots\\
&\hspace{4cm} \big\{\big[ \mathbf{Y}_{1,i_{2r-1}}(U)-\mathbf{Y}_{-1,i_{2r-1}}(U)\big]\big[ \mathbf{Y}_{1,i_{2r}}(U^\ast)-\mathbf{Y}_{-1,i_{2r}}(U^\ast)\big]\big\} \big)\\
&=(-1)^r \sum_{\si\in \{-1,1\}^{2r}} (-1)^{N(\si)} \sum_{i_1,\ldots ,i_{2r}} \vp\big( \big( \mathbf{Y}_{\si_1,i_1}[U]\mathbf{Y}_{\si_2,i_2}[U^\ast] \big) \cdots \big( \mathbf{Y}_{\si_{2r-1},i_{2r-1}}[U]\mathbf{Y}_{\si_{2r},i_{2r}}[U^\ast] \big) \big) \ ,
\end{align*}
where $N(\si)$ denotes the number of $(-1)$ in $\si$. At this point, recall that $(Y_i,U_j)$ is a semicircular family, and so, by (\ref{form-wick}), we can go ahead with our expansion and write the previous quantity as
\begin{multline}\label{interm-quant}
(-1)^r \sum_{\si\in \{-1,1\}^{2r}} (-1)^{N(\si)}\\ \sum_{i_1,\ldots ,i_{2r}}\sum_{\pi \in NC_2(2r(m+2))}
 \ka_\pi\big( \big(\mathbb{Y}_{\si_1,i_1}[\mathbb{U}],\mathbb{Y}_{\si_2,i_2}[\mathbb{U}^\ast]\big),\ldots, \big(\mathbb{Y}_{\si_{2r-1},i_{2r-1}}[\mathbb{U}],\mathbb{Y}_{\si_{2r},i_{2r}}[\mathbb{U}^\ast]\big)\big) \ ,
\end{multline}
where we have set $\mathbb{U}:=(U_1,\ldots,U_m)$ and $\mathbb{U}^\ast:=(U_m,\ldots,U_1)$.

\smallskip

Let us now consider the subset of $NC_2(2r(m+2))$ given by 
$$\mathcal{E}_1:=\{\pi \in NC_2(2r(m+2)): \, (1,m+2)\in \pi\} \ .$$
With expression (\ref{interm-quant}) in mind, $\ce_1$ thus corresponds to the set of pairings $\pi$ for which the variables $Y_{2i_1}$ and $Y_{2i_1+1}$ are \enquote{connected} within $\ka_\pi(...)$ (recall notation (\ref{notation:vp-pi})). Observe in particular that, for fixed $\si$ and $i_1,\ldots,i_{2r}$,
\begin{align*}
&\sum_{\pi \in \ce_1}  \ka_\pi\big( \big(\mathbb{Y}_{1,i_1}[\mathbb{U}],\mathbb{Y}_{\si_2,i_2}[\mathbb{U}^\ast]\big),\ldots, \big(\mathbb{Y}_{\si_{2r-1},i_{2r-1}}[\mathbb{U}],\mathbb{Y}_{\si_{2r},i_{2r}}[\mathbb{U}^\ast]\big)\big)\\
&=\vp\big( Y_{2i_1}Y_{2i_1+1}\big)\vp\big( U\big)\\
&\hspace{0.5cm} \sum_{\pi \in NC_2((2r-1)(m+2))} \ka_\pi\big(\mathbb{Y}_{\si_2,i_2}[\mathbb{U}^\ast],\big(\mathbb{Y}_{\si_3,i_3}[\mathbb{U}],\mathbb{Y}_{\si_4,i_4}[\mathbb{U}^\ast]\big),\ldots, \big(\mathbb{Y}_{\si_{2r-1},i_{2r-1}}[\mathbb{U}],\mathbb{Y}_{\si_{2r},i_{2r}}[\mathbb{U}^\ast]\big)\big) \\
&=\sum_{\pi \in \ce_1}  \ka_\pi\big( \big(\mathbb{Y}_{-1,i_1}[\mathbb{U}],\mathbb{Y}_{\si_2,i_2}[\mathbb{U}^\ast]\big),\ldots, \big(\mathbb{Y}_{\si_{2r-1},i_{2r-1}}[\mathbb{U}],\mathbb{Y}_{\si_{2r},i_{2r}}[\mathbb{U}^\ast]\big)\big) \ ,
\end{align*}
and as result
\begin{align*}
&\sum_{\si\in \{-1,1\}^{2r}} (-1)^{N(\si)} \sum_{i_1,\ldots ,i_{2r}}\sum_{\pi \in \ce_1}
 \ka_\pi\big( \big(\mathbb{Y}_{\si_1,i_1}[\mathbb{U}],\mathbb{Y}_{\si_2,i_2}[\mathbb{U}^\ast]\big),\ldots, \big(\mathbb{Y}_{\si_{2r-1},i_{2r-1}}[\mathbb{U}],\mathbb{Y}_{\si_{2r},i_{2r}}[\mathbb{U}^\ast]\big)\big)\\
&\hspace{2cm}=\sum_{\substack{\si\in \{-1,1\}^{2r}\\ \si_1=1}} \ldots+\sum_{\substack{\si\in \{-1,1\}^{2r}\\ \si_1=-1}}\ldots \ = \ 0\ .\\
\end{align*}
Along the same idea, consider the subset 
$$\ce_2:=\{\pi \in NC_2(2r(m+2)): \, (1,m+2) \notin \pi \, , \, (m+3,2(m+2))\in \pi\} \ ,$$
so that, just as above,
\begin{align*}
&\sum_{\pi \in \ce_2}  \ka_\pi\big( \big(\mathbb{Y}_{\si_1,i_1}[\mathbb{U}],\mathbb{Y}_{1,i_2}[\mathbb{U}^\ast]\big),\ldots, \big(\mathbb{Y}_{\si_{2r-1},i_{2r-1}}[\mathbb{U}],\mathbb{Y}_{\si_{2r},i_{2r}}[\mathbb{U}^\ast]\big)\big)\\
&=\vp\big( Y_{2i_2}Y_{2i_2+1}\big)\vp\big( U^\ast\big)\\
&\hspace{0.5cm} \sum_{\pi \in \tilde{\ce}_2} \ka_\pi\big(\mathbb{Y}_{\si_1,i_1}[\mathbb{U}],\big(\mathbb{Y}_{\si_3,i_3}[\mathbb{U}],\mathbb{Y}_{\si_4,i_4}[\mathbb{U}^\ast]\big),\ldots, \big(\mathbb{Y}_{\si_{2r-1},i_{2r-1}}[\mathbb{U}],\mathbb{Y}_{\si_{2r},i_{2r}}[\mathbb{U}^\ast]\big)\big) \\
&=\sum_{\pi \in \ce_2}  \ka_\pi\big( \big(\mathbb{Y}_{\si_1,i_1}[\mathbb{U}],\mathbb{Y}_{-1,i_2}[\mathbb{U}^\ast]\big),\ldots, \big(\mathbb{Y}_{\si_{2r-1},i_{2r-1}}[\mathbb{U}],\mathbb{Y}_{\si_{2r},i_{2r}}[\mathbb{U}^\ast]\big)\big) \ ,
\end{align*}
where $\tilde{\ce}_2:=\{\pi \in NC_2((2r-1)(m+2)): \, (1,m+2) \notin \pi\}$, and accordingly
\begin{align*}
&\sum_{\si\in \{-1,1\}^{2r}} (-1)^{N(\si)} \sum_{i_1,\ldots ,i_{2r}}\sum_{\pi \in \ce_2}
 \ka_\pi\big( \big(\mathbb{Y}_{\si_1,i_1}[\mathbb{U}],\mathbb{Y}_{\si_2,i_2}[\mathbb{U}^\ast]\big),\ldots, \big(\mathbb{Y}_{\si_{2r-1},i_{2r-1}}[\mathbb{U}],\mathbb{Y}_{\si_{2r},i_{2r}}[\mathbb{U}^\ast]\big)\big)\\
&\hspace{2cm}=\sum_{\substack{\si\in \{-1,1\}^{2r}\\ \si_2=1}} \ldots+\sum_{\substack{\si\in \{-1,1\}^{2r}\\ \si_2=-1}}\ldots \ = \ 0 \ .\\
\end{align*}
Iterating the procedure, we see that the quantity (\ref{interm-quant}) reduces in fact to 
\begin{align}
&(-1)^r \sum_{\pi \in \ce}\sum_{\si\in \{-1,1\}^{2r}} (-1)^{N(\si)}\nonumber\\
&\hspace{1.5cm}\sum_{i_1,\ldots ,i_{2r}}
 \ka_\pi\big( \big(\mathbb{Y}_{\si_1,i_1}[\mathbb{U}],\mathbb{Y}_{\si_2,i_2}[\mathbb{U}^\ast]\big),\ldots, \big(\mathbb{Y}_{\si_{2r-1},i_{2r-1}}[\mathbb{U}],\mathbb{Y}_{\si_{2r},i_{2r}}[\mathbb{U}^\ast]\big)\big) \ ,\label{interm-quant-2}
\end{align}
where 
$$\ce:=\{\pi\in NC_2(2r(m+2)): \, \text{for all} \ j=1,\ldots,2r \, , \, ((j-1)(m+2)+1,j(m+2))\notin \pi\} \ .$$

\

As a next step, observe that for each fixed $\pi\in \ce$ and $\si\in \{-1,1\}^{2r}$, the sum
$$\sum_{i_1,\ldots ,i_{2r}}
 \ka_\pi\big( \big(\mathbb{Y}_{\si_1,i_1}[\mathbb{U}],\mathbb{Y}_{\si_2,i_2}[\mathbb{U}^\ast]\big),\ldots, \big(\mathbb{Y}_{\si_{2r-1},i_{2r-1}}[\mathbb{U}],\mathbb{Y}_{\si_{2r},i_{2r}}[\mathbb{U}^\ast]\big)\big)$$
can always be written as a product of three terms $P_i$ of the form
$$
P_1=\prod_{j=1}^{q_1} \bigg( \sum_{i_1,\ldots,i_{p_j}} \vp\big(Z_{(j,1),i_1}Z'_{(j,2),i_2}\big) \vp\big(Z_{(j,2),i_2}Z'_{(j,3),i_3}\big)\cdots \vp\big(Z_{(j,p_j),i_{p_j}}Z'_{(j,1),i_1}\big) \bigg) \ ,
$$
$$
P_2=\prod_{j=1}^{q_2} \bigg( \sum_{i_1,\ldots,i_{r_j}} \vp\big(U_{\al_j} W'_{(j,1),i_1}\big)\vp\big(W_{(j,1),i_1}W'_{(j,2),i_2}\big) \vp\big(W_{(j,2),i_2}W'_{(j,3),i_3}\big)\cdots
 \vp\big(W_{(j,r_j),i_{r_j}}U_{\beta_j}\big) \bigg) \ ,
$$
$$
P_3=\prod_{j=1}^{q_3} \vp\big( U_{\eta_j}U_{\la_j}\big) \ ,
$$
where: 

\smallskip

\noindent
$(a)$ the integers $q_1,q_2,q_3,p_j,r_j$ are such that $(p_1+\ldots+p_{q_1})+(r_1+\ldots+r_{q_2})=2r$ and $2q_2+2q_3=2mr$;

\smallskip

\noindent
$(b)$ the variables $(Z,Z')$ (resp. $(W,W')$) are such that for all $j=1,\ldots,q_1$, $p=1,\ldots,p_j$ (resp. $j=1,\ldots,q_2$, $p=1,\ldots,r_j$) and $i=k,\ldots,\ell-1$, $\{Z_{(j,p),i},Z'_{(j,p),i}\}=\{Y_{2i},Y_{2i+1}\}$ (resp. $\{W_{(j,p),i},W'_{(j,p),i}\}=\{Y_{2i},Y_{2i+1}\}$).
\smallskip

\noindent
$(c)$ the coefficients $1\leq \al_j,\beta_j,\eta_j,\la_j \leq m$ are such that each variable $U_\al$ ($1\leq \al \leq m$) appears exactly $2r$-times in the product $P_2 P_3$.

\

Let us now bound the product $P_1P_2P_3$ in each of the two above-described situations for $U$.

\

\noindent
\textbf{Situation A.}

\smallskip

\noindent
As far as $P_1$ is concerned, we can first use the subsequent elementary Lemma \ref{lem:CS} to assert that for each $j=1,\ldots,q_1$,
\begin{align*}
&\sum_{i_1,\ldots,i_{p_j}} \vp\big(Z_{(j,1),i_1}Z'_{(j,2),i_2}\big) \vp\big(Z_{(j,2),i_2}Z'_{(j,3),i_3}\big)\cdots \vp\big(Z_{(j,p_j),i_{p_j}}Z'_{(j,1),i_1}\big)\\
&\leq \bigg[ \prod_{q=1}^{p_j-1} \bigg( \sum_{i_1,i_2} \vp\big(Z_{(j,q),i_1}Z'_{(j,q+1),i_2}\big)^2 \bigg)\bigg]^{1/2}\bigg( \sum_{i_1,i_2} \vp\big(Z_{(j,p_j),i_1}Z'_{(j,1),i_2}\big)^2 \bigg)^{1/2} \ .
\end{align*}
We are here in a position to apply Lemma \ref{lem:estim-cova-1} below and deduce that for each $j=1,\ldots,q_1$,
$$\sum_{i_1,\ldots,i_{p_j}} \vp\big(Z_{(j,1),i_1}Z'_{(j,2),i_2}\big) \vp\big(Z_{(j,2),i_2}Z'_{(j,3),i_3}\big)\cdots \vp\big(Z_{(j,p_j),i_{p_j}}Z'_{(j,1),i_1}\big) \lesssim \frac{|t-s|^{(2H-\varepsilon)p_j}}{2^{n\varepsilon p_j}} \ ,$$
and therefore
\begin{equation}\label{p-1}
P_1 \lesssim \frac{|t-s|^{(2H-\varepsilon)(p_1+\ldots+p_{q_1})}}{2^{n\varepsilon (p_1+\ldots+p_{q_1})}} \ .
\end{equation}

\

In order to estimate $P_2$, let us first write, with the help of Lemma \ref{lem:CS}, and for each $j=1,\ldots,q_2$,
\begin{align}
&\sum_{i_1,\ldots,i_{r_j}} \vp\big(U_{\al_j} W'_{(j,1),i_1}\big)\vp\big(W_{(j,1),i_1}W'_{(j,2),i_2}\big) \vp\big(W_{(j,2),i_2}W'_{(j,3),i_3}\big)\cdots
 \vp\big(W_{(j,r_j),i_{r_j}}U_{\beta_j}\big) \nonumber\\
&\leq \bigg( \sum_{i'_1,i'_{r_j}} \vp\big(U_{\al_j} W'_{(j,1),i'_1}\big)^2 \vp\big(W_{(j,r_j),i'_{r_j}}U_{\beta_j}\big)^2 \bigg)^{1/2}\nonumber\\
&\hspace{0.5cm} \bigg(\sum_{i_1,i_{r_j}} \bigg[ \sum_{i_2,\ldots,i_{r_{j-1}}}\vp\big(W_{(j,1),i_1}W'_{(j,2),i_2}\big) \vp\big(W_{(j,2),i_2}W'_{(j,3),i_3}\big)\cdots \vp\big(W_{(j,r_j-1),i_{r_j-1}}W'_{(j,r_j),i_{r_j}}\big)\bigg]^2 \bigg)^{1/2}\nonumber\\
&\leq \bigg( \sum_{i'_1} \vp\big(U_{\al_j} W'_{(j,1),i'_1}\big)^2\bigg)^{1/2} \bigg(\sum_{i'_{r_j}}\vp\big(W_{(j,r_j),i'_{r_j}}U_{\beta_j}\big)^2 \bigg)^{1/2}\bigg[ \prod_{q=1}^{r_j-1} \bigg( \sum_{i_1,i_2} \vp\big(W_{(j,q),i_1}W'_{(j,q+1),i_2}\big)^2 \bigg)\bigg]^{1/2} \ .\label{p-2-inter}
\end{align}
We can now combine the results of the subsequent Lemmas \ref{lem:estim-cova-1} and \ref{lem:estim-cova-2} to deduce that for each $j=1,\ldots,q_2$,
\begin{align*}
&\sum_{i_1,\ldots,i_{r_j}} \vp\big(U_{\al_j} W'_{(j,1),i_1}\big)\vp\big(W_{(j,1),i_1}W'_{(j,2),i_2}\big) \vp\big(W_{(j,2),i_2}W'_{(j,3),i_3}\big)\cdots
 \vp\big(W_{(j,r_j),i_{r_j}}U_{\beta_j}\big)\\
&\hspace{4cm}\lesssim |v_{\al_j}-u_{\al_j}|^{H}|v_{\beta_j}-u_{\beta_j}|^{H}\frac{|t-s|^{(2H-\varepsilon)r_j}}{2^{n\varepsilon r_j}} \ ,
\end{align*}
and as a result
\begin{equation}\label{p-2}
P_2 \lesssim \bigg( \prod_{j=1}^{q_2} |v_{\al_j}-u_{\al_j}|^{H}|v_{\beta_j}-u_{\beta_j}|^{H} \bigg) \frac{|t-s|^{(2H-\varepsilon)(r_1+\ldots+r_{q_2})}}{2^{n\varepsilon (r_1+\ldots+r_{q_2})}} \ .
\end{equation}
Finally, the estimation of $P_3$ is an immediate consequence of (\ref{holder-regu}):
\begin{equation}\label{p-3}
P_3\leq \prod_{j=1}^{q_3} \|U_{\eta_j}\| \|U_{\la_j}\| \lesssim \prod_{j=1}^{q_3} |v_{\eta_j}-u_{\eta_j}|^{H}|v_{\la_j}-u_{\la_j}|^{H} \ .
\end{equation}
Combining (\ref{p-1})-(\ref{p-2})-(\ref{p-3}) with the above constraints $(a)$-$(b)$-$(c)$, we obtain that for each fixed $\pi\in \ce$ and $\si\in \{-1,1\}^{2r}$,
\begin{align*}
&\sum_{i_1,\ldots ,i_{2r}}
 \ka_\pi\big( \big(\mathbb{Y}_{\si_1,i_1}[\mathbb{U}],\mathbb{Y}_{\si_2,i_2}[\mathbb{U}^\ast]\big),\ldots, \big(\mathbb{Y}_{\si_{2r-1},i_{2r-1}}[\mathbb{U}],\mathbb{Y}_{\si_{2r},i_{2r}}[\mathbb{U}^\ast]\big)\big)\\
&\hspace{3cm}\lesssim \frac{|t-s|^{2r(2H-\varepsilon)}}{2^{2r n\varepsilon}} \prod_{j=1}^m |v_j-u_j|^{2rH} \ .
\end{align*}
Going back to (\ref{interm-quant-2}), we have thus shown that for every $r\geq 1$,
\begin{align*}
&\vp\bigg( \bigg( \bigg(\sum_{i=k}^{\ell-1} \big[  Y_{2i}  UY_{2i+1}-  Y_{2i+1}U Y_{2i}\big]\bigg) \bigg(\sum_{i=k}^{\ell-1} \big[  Y_{2i}U  Y_{2i+1}-  Y_{2i+1}U Y_{2i}\big]\bigg)^{\!\!\! \ast} \bigg)^r \bigg)^{1/(2r)}\\
&\hspace{2cm}\lesssim \big(|NC_2(2r(m+2))|^{1/(2r(m+2))}\big)^{m+2} \frac{|t-s|^{(2H-\varepsilon)}}{2^{n\varepsilon}} \prod_{j=1}^m |v_j-u_j|^{H}\ .
\end{align*}
By letting $r$ tend to infinity, we get the desired estimate, namely
\begin{equation}\label{estim-gen}
\big\|\big\{ \mathbb{X}^{2,(n+1)}_{t_k^n t_\ell^n}-\mathbb{X}^{2,(n)}_{t_k^n t_\ell^n}\big\}\big[ \der X_{u_1 v_1} \cdots \der X_{u_m v_m}\big]\big\|\lesssim 2^m\frac{|t-s|^{(2H-\varepsilon)}}{2^{n\varepsilon}} \prod_{j=1}^m |v_j-u_j|^{H}\ .
\end{equation}
It is readily checked that the above procedure can also be applied in the case $m=0$, yielding
\begin{equation}\label{estim-part}
\big\|\big\{ \mathbb{X}^{2,(n+1)}_{t_k^n t_\ell^n}-\mathbb{X}^{2,(n)}_{t_k^n t_\ell^n}\big\}\big[ 1\big]\big\| \lesssim \frac{|t-s|^{(2H-\varepsilon)}}{2^{n\varepsilon}}\ .
\end{equation}

\

\noindent
\textbf{Situation B.}

\smallskip

\noindent
The expression of $P_1$ is of course the same as in Situation A, and thus, just as above, we have
\begin{equation}\label{p-1-b}
P_1 \lesssim \frac{|t-s|^{(2H-\varepsilon)(p_1+\ldots+p_{q_1})}}{2^{n\varepsilon (p_1+\ldots+p_{q_1})}} \ .
\end{equation}
In order to estimate $P_2$, let us first write, just as in (\ref{p-2-inter}), and for each $j=1,\ldots,q_2$,
\begin{align*}
&\sum_{i_1,\ldots,i_{r_j}} \vp\big(U_{\al_j} W'_{(j,1),i_1}\big)\vp\big(W_{(j,1),i_1}W'_{(j,2),i_2}\big) \vp\big(W_{(j,2),i_2}W'_{(j,3),i_3}\big)\cdots
 \vp\big(W_{(j,r_j),i_{r_j}}U_{\beta_j}\big) \\
&\leq \bigg( \sum_{i'_1} \vp\big(U_{\al_j} W'_{(j,1),i'_1}\big)^2\bigg)^{1/2} \bigg(\sum_{i'_{r_j}}\vp\big(W_{(j,r_j),i'_{r_j}}U_{\beta_j}\big)^2 \bigg)^{1/2}\bigg[ \prod_{q=1}^{r_j-1} \bigg( \sum_{i_1,i_2} \vp\big(W_{(j,q),i_1}W'_{(j,q+1),i_2}\big)^2 \bigg)\bigg]^{1/2} \ .
\end{align*}
We can now combine Lemma \ref{lem:estim-cova-1} and Corollary \ref{cor:estim-cova-2-gene} below to obtain that for each $j=1,\ldots,q_2$,
\begin{align*}
&\sum_{i_1,\ldots,i_{r_j}} \vp\big(U_{\al_j} W'_{(j,1),i_1}\big)\vp\big(W_{(j,1),i_1}W'_{(j,2),i_2}\big) \vp\big(W_{(j,2),i_2}W'_{(j,3),i_3}\big)\cdots
 \vp\big(W_{(j,r_j),i_{r_j}}U_{\beta_j}\big)\\
&\hspace{9cm}\lesssim \frac{|t-s|^{(2H-\varepsilon)r_j}}{2^{n\varepsilon r_j}}Q_{\al_j,\be_j} \ ,
\end{align*}
with
\begin{align*}
Q_{\al_j,\be_j}:=&\mathbf{1}_{\{\al_j\neq \iota,\beta_j\neq \iota\}} |v_{\al_j}-u_{\al_j}|^{H}|v_{\beta_j}-u_{\beta_j}|^{H} +\mathbf{1}_{\{\al_j=\iota,\beta_j= \iota\}} \frac{|v_{\iota}-u_{\iota}|^{2H-2\varepsilon'}}{2^{2n_\iota \varepsilon'}}\\
&+\mathbf{1}_{\{\al_j=\iota,\beta_j\neq \iota\}} \frac{|v_{\iota}-u_{\iota}|^{H-\varepsilon'}}{2^{n_\iota \varepsilon'}}|v_{\beta_j}-u_{\beta_j}|^{H}+\mathbf{1}_{\{\al_j\neq\iota,\beta_j= \iota\}} |v_{\al_j}-u_{\al_j}|^{H}\frac{|v_{\iota}-u_{\iota}|^{H-\varepsilon'}}{2^{n_\iota \varepsilon'}} \ ,
\end{align*}
and accordingly
\begin{equation}\label{p-2-b}
P_2 \lesssim \bigg( \prod_{j=1}^{q_2} Q_{\al_j,\be_j} \bigg) \frac{|t-s|^{(2H-\varepsilon)(r_1+\ldots+r_{q_2})}}{2^{n\varepsilon (r_1+\ldots+r_{q_2})}} \ .
\end{equation}
Finally, the estimation of $P_3$ in this situation follows from the two controls in (\ref{fir-or-pr}):
\begin{equation}\label{p-3-b}
P_3\leq \prod_{j=1}^{q_3} \|U_{\eta_j}\| \|U_{\la_j}\|\lesssim \prod_{j=1}^{q_3} Q_{\eta_j,\la_j} \ .
\end{equation}
Combining (\ref{p-1-b})-(\ref{p-2-b})-(\ref{p-3-b}) with the above constraints $(a)$-$(b)$-$(c)$, we deduce, for each fixed $\pi\in \ce$ and $\si\in \{-1,1\}^{2r}$,
\begin{align*}
&\sum_{i_1,\ldots ,i_{2r}}
 \ka_\pi\big( \big(\mathbb{Y}_{\si_1,i_1}[\mathbb{U}],\mathbb{Y}_{\si_2,i_2}[\mathbb{U}^\ast]\big),\ldots, \big(\mathbb{Y}_{\si_{2r-1},i_{2r-1}}[\mathbb{U}],\mathbb{Y}_{\si_{2r},i_{2r}}[\mathbb{U}^\ast]\big)\big)\\
&\hspace{3cm}\lesssim \frac{|t-s|^{2r(2H-\varepsilon)}}{2^{2r n\varepsilon}}\frac{|v_\iota-u_\iota|^{2r(H-\varepsilon')}}{2^{2r N_{\iota}\varepsilon'}} \prod_{\substack{j=1\\j\neq \iota}}^m |v_j-u_j|^{2rH} \ ,
\end{align*}
and we can then use the same arguments as in Situation A to derive that
\begin{align}
&\big\|\big\{ \mathbb{X}^{2,(n+1)}_{t_k^n t_\ell^n}-\mathbb{X}^{2,(n)}_{t_k^n t_\ell^n}\big\}\big[ \der X^{(N_1)}_{u_1 v_1} \cdots \der(X^{(N_\iota)}-X)_{u_\iota v_\iota} \cdots \der X^{(N_m)}_{u_m v_m}\big]\big\|\nonumber\\
&\hspace{2cm}\lesssim 2^m \frac{|t-s|^{2H-\varepsilon}}{2^{n\varepsilon}}\frac{|v_\iota-u_\iota|^{H-\varepsilon'}}{2^{N_\iota\varepsilon'}} \prod_{\substack{j=1,\ldots,m\\ j\neq \iota}} |u_j-v_j|^H\ ,\label{estim-gen-b}
\end{align}
for some proportional constant independent of $n,N_1,\ldots,N_m$.

\

\noindent
\textbf{Conclusion.}

\smallskip

First, based on (\ref{trivial-case}), (\ref{reduction}), (\ref{estim-gen}) and (\ref{estim-part}), we can assert that for all fixed $0\leq s\leq t\leq 1$ and $U\in \ca_s$, $(\mathbb{X}^{2,(n)}_{st}[U])_{n\geq 1}$ is a Cauchy sequence in $(\ca,\|.\|)$, and therefore it converges to an element $\mathbb{X}^{2}_{st}[U]$, as desired.

\smallskip

The fact that $\mathbb{X}_{st}^{2}$ is linear (as a function of $U$) follows immediately from the linearity of $\mathbb{X}_{st}^{2,(n)}$, and in the same way, identity (\ref{chen}) is a straightforward consequence of the (readily-checked) relation
$$\mathbb{X}^{2,(n)}_{st}[U]-\mathbb{X}^{2,(n)}_{su}[U]-\mathbb{X}^{2,(n)}_{ut}[U]=\der X^{(n)}_{su} U \der X^{(n)}_{ut} \ .$$
Finally, estimate (\ref{roughness-general-1}), resp. estimate (\ref{roughness-general-2}), follows at once from (\ref{trivial-case}), (\ref{reduction}) and (\ref{estim-gen}), resp. (\ref{trivial-case}), (\ref{reduction}) and (\ref{estim-gen-b}).

\end{proof}

We are now left with the proof of the few technical results related to the control of the covariances.

\begin{lemma}\label{lem:CS}
Given a finite set $I$, an integer $p\geq 1$ and real quantities $A^{(q)}_{i_1i_2}$ ($1\leq q\leq p$, $i_1,i_2\in I$), it holds that
$$\sum_{i_1,i_p\in I} \bigg[ \sum_{i_2,\ldots,i_{p-1}\in I} A^{(1)}_{i_1 i_2} A^{(2)}_{i_2 i_3} \cdots A^{(p-1)}_{i_{p-1}i_p}\bigg]^2\leq \prod_{q=1}^{p-1} \bigg(\sum_{i_1,i_2\in I} \big(A^{(q)}_{i_1i_2}\big)^2\bigg)\ ,$$
and as a particular consequence
$$\sum_{i_1,\ldots,i_{p}\in I} A^{(1)}_{i_1 i_2} A^{(2)}_{i_2 i_3} \cdots A^{(p-1)}_{i_{p-1}i_p} A^{(p)}_{i_p i_1} \leq \prod_{q=1}^{p}\bigg(\sum_{i_1,i_2\in I} \big(A^{(q)}_{i_1i_2}\big)^2\bigg)^{1/2} \ . $$
\end{lemma}

\begin{proof}
These results can actually be shown through an easy iteration of Cauchy-Schwarz inequality (we have labeled them for the sake of clarity only). 
\end{proof}

\begin{lemma}\label{lem:estim-cova-1}
With the notations of the above proof, one has, for every $\varepsilon \in (0,2H-\frac12)$,
$$\max\bigg( \sum_{i,j=k}^{\ell-1} \vp\big(Y_{2i}Y_{2j}\big)^2,\sum_{i,j=k}^{\ell-1} \vp\big(Y_{2i}Y_{2j+1}\big)^2,\sum_{i,j=k}^{\ell-1} \vp\big(Y_{2i+1}Y_{2j+1}\big)^2\bigg) \lesssim \frac{|t-s|^{4H-2\varepsilon}}{2^{2n\varepsilon}} \ .$$
\end{lemma}
\begin{proof}
Let us naturally write 
$$\sum_{i,j=k}^{\ell-1} \vp\big(Y_{2i}Y_{2j+1}\big)^2=\sum_{i=k}^{\ell-1} \vp\big(Y_{2i}Y_{2i+1}\big)^2+\sum_{\substack{i,j=k\\i<j}}^{\ell-1} \vp\big(Y_{2i}Y_{2j+1}\big)^2+\sum_{\substack{i,j=k\\j<i}}^{\ell-1} \vp\big(Y_{2i}Y_{2j+1}\big)^2 \ .$$
First, one has obviously (denoting by $c_H$ any positive constant that only depends on $H$)
$$\sum_{i=k}^{\ell-1} \vp\big(Y_{2i}Y_{2i+1}\big)^2=\sum_{i=k}^{\ell-1} \frac{c_H}{2^{4Hn}}=c_H \frac{\ell-k}{2^{4Hn}}\leq c_H \frac{|t-s|}{2^{n(4H-1)}} \leq c_H \frac{|t-s|^{4H-2\varepsilon}}{2^{2n\varepsilon}} \ ,$$
where we have used the fact that $2\varepsilon<4H-1$ and $2^{-n} \leq |t-s|$ (by (\ref{choice-s-t})) to get the last inequality.

\smallskip

Then write
\begin{align*}
\sum_{\substack{i,j=k\\i<j}}^{\ell-1} \vp\big(Y_{2i}Y_{2j+1}\big)^2 &=c_H \sum_{j=k}^{\ell-1}\sum_{i=k}^{j-1} \bigg( \int_{t_{2j+1}^{n+1}}^{t_{2j+2}^{n+1}} d\tau_1 \int_{t_{2i}^{n+1}}^{t_{2i+1}^{n+1}} \frac{d\tau_2}{(\tau_1-\tau_2)^{2-2H}} \bigg)^2\\
&\leq c_H \sum_{j=k}^{\ell-1}\sum_{i=k}^{j-1} \bigg( \int_{t_{j}^n}^{t_{j+1}^n} d\tau_1 \int_{t_{i}^n}^{t_{i+1}^n} \frac{d\tau_2}{(\tau_1-\tau_2)^{2-2H}} \bigg)^2 \ .
\end{align*}
Using elementary changes of variables, we can easily rewrite the latter quantity as
$$\sum_{j=k}^{\ell-1}\sum_{i=k}^{j-1} \bigg( \int_{t_{j}^n}^{t_{j+1}^n} d\tau_1 \int_{t_{i}^n}^{t_{i+1}^n} \frac{d\tau_2}{(\tau_1-\tau_2)^{2-2H}} \bigg)^2=2^{-4Hn} \sum_{j=1}^{\ell-k}\sum_{i=1}^{j-1} \bigg( \int_0^1 d\tau_1 \int_{i-1}^{i} \frac{d\tau_2}{(\tau_1+\tau_2)^{2-2H}}\bigg)^2 \ ,$$
which yields 
\begin{eqnarray*}
\sum_{\substack{i,j=k\\i<j}}^{\ell-1} \vp\big(Y_{2i}Y_{2j+1}\big)^2&\leq& \frac{c_H}{2^{4Hn}} \sum_{j=1}^{\ell-k}\sum_{i=1}^{j-1} \bigg( \int_0^1 \frac{d\tau_1}{\tau_1^{1-\varepsilon}} \int_{i-1}^{i} \frac{d\tau_2}{\tau_2^{1-2H+\varepsilon}}\bigg)^2\\
&\leq &\frac{c_{H,\varepsilon} }{2^{4Hn}} \sum_{j=1}^{\ell-k}\sum_{i=1}^{j-1} \int_{i-1}^{i} \frac{d\tau_2}{\tau_2^{2-4H+2\varepsilon}}\ \leq \ \frac{c_{H,\varepsilon} }{2^{4Hn}} \sum_{j=1}^{\ell-k} \int_{0}^{j} \frac{d\tau_2}{\tau_2^{2-4H+2\varepsilon}}\\
& \leq &  \frac{c_{H,\varepsilon} }{2^{4Hn}} \sum_{j=1}^{\ell-k} j^{4H-2\varepsilon-1} \ \leq \ c_{H,\varepsilon}  \frac{|\ell-k|^{4H-2\varepsilon}}{2^{4Hn}} \ \leq \ c_{H,\varepsilon}  \frac{|t-s|^{4H-2\varepsilon}}{2^{2n\varepsilon}} \ .
\end{eqnarray*}
The same arguments can of course be used to bound $\sum_{i,j=k}^{\ell-1} 1_{\{j<i\}}\vp\big(Y_{2i}Y_{2j+1}\big)^2$, so that
$$\sum_{i,j=k}^{\ell-1} \vp\big(Y_{2i}Y_{2j+1}\big)^2 \leq c_{H,\varepsilon} \frac{|t-s|^{4H-2\varepsilon}}{2^{2n\varepsilon}} \ .$$
We can then estimate $\sum_{i,j=k}^{\ell-1} \vp\big(Y_{2i}Y_{2j}\big)^2$ and $\sum_{i,j=k}^{\ell-1} \vp\big(Y_{2i+1}Y_{2j+1}\big)^2$ along the same procedure.
\end{proof}

\begin{lemma}\label{lem:estim-cova-2}
With the notations of the above proof, one has, for all $j=1,\ldots,m$ and $\varepsilon \in [0,H)$,
\begin{equation}\label{estim-cova-2}
\max\bigg( \sum_{i=k}^{\ell-1} \vp\big( \der X_{u_j v_j} Y_{2i}\big)^2\ ,\ \sum_{i=k}^{\ell-1} \vp\big(\der X_{u_j v_j}Y_{2i+1}\big)^2\bigg) \lesssim |v_j-u_j|^{2H}\frac{|t-s|^{2H-\varepsilon}}{2^{n\varepsilon}} \ .
\end{equation}
\end{lemma}
\begin{proof}
Recall that for all $j=1,\ldots,m$ and $i=k,\ldots,\ell-1$, one has $0\leq u_j\leq v_j \leq s \leq t_{2i}^{n+1}\leq t_{2i+1}^{n+1}$, and so
\begin{align*}
&\sum_{i=k}^{\ell-1} \vp\big( \der X_{u_j v_j} Y_{2i}\big)^2=c_H \sum_{i=k}^{\ell-1} \bigg( \int_{t_{2i}^{n+1}}^{t_{2i+1}^{n+1}} d\tau_1 \int_{u_j}^{v_j} \frac{d\tau_2}{(\tau_1-\tau_2)^{2-2H}} \bigg)^2\\
&\leq c_H \sum_{i=k}^{\ell-1} \bigg( \int_{t_{2i}^{n+1}}^{t_{2i+1}^{n+1}} \frac{d\tau_1}{(\tau_1-s)^{1-H}} \int_{u_j}^{v_j} \frac{d\tau_2}{(v_j-\tau_2)^{1-H}} \bigg)^2\\
&\leq c_H |v_j-u_j|^{2H} \sum_{i=k}^{\ell-1} \bigg( \int_{t_{2i}^{n+1}}^{t_{2i+1}^{n+1}} \frac{d\tau_1}{(\tau_1-s)^{1-H}}\bigg) \bigg( \int_{t_{2i}^{n+1}}^{t_{2i+1}^{n+1}} \frac{d\tau_1}{(\tau_1-t_{2i}^{n+1})^{1-H}} \bigg)^{\frac{\varepsilon}{H}}\bigg( \int_s^t \frac{d\tau_1}{(\tau_1-s)^{1-H}} \bigg)^{1-\frac{\varepsilon}{H}}\\
&\leq c_H |v_j-u_j|^{2H} \frac{|t-s|^{H-\varepsilon}}{2^{n\varepsilon}} \sum_{i=k}^{\ell-1} \int_{t_{2i}^{n+1}}^{t_{2i+1}^{n+1}} \frac{d\tau_1}{(\tau_1-s)^{1-H}}\ \leq \  c_H |v_j-u_j|^{2H} \frac{|t-s|^{2H-\varepsilon}}{2^{n\varepsilon}} \ .\\
\end{align*}
The same arguments apply of course to $\sum_{i=k}^{\ell-1} \vp\big( \der X_{u_j v_j} Y_{2i+1}\big)^2$.
\end{proof}

\begin{corollary}\label{cor:estim-cova-2-gene}
With the notations of the above proof, one has, for all $j=1,\ldots,m$ and $\varepsilon,\varepsilon' \in [0,H)$,
\begin{equation}\label{estim-cova-2-approx}
\max\bigg( \sum_{i=k}^{\ell-1} \vp\big( \der X^{(N_j)}_{u_j v_j} Y_{2i}\big)^2\ ,\ \sum_{i=k}^{\ell-1} \vp\big(\der X^{(N_j)}_{u_j v_j}Y_{2i+1}\big)^2\bigg) \lesssim |v_j-u_j|^{2H}\frac{|t-s|^{2H-\varepsilon}}{2^{n\varepsilon}}
\end{equation}
and
\begin{align}\label{estim-cova-2-approx-2}
&\max\bigg( \sum_{i=k}^{\ell-1} \vp\big( \der \{X^{(N_j)}-X\}_{u_j v_j} Y_{2i}\big)^2\ ,\ \sum_{i=k}^{\ell-1} \vp\big(\der \{X^{(N_j)}-X\}_{u_j v_j}Y_{2i+1}\big)^2\bigg)\nonumber\\
&\hspace{6cm} \lesssim \frac{|v_j-u_j|^{2H-\varepsilon'}}{2^{N_j \varepsilon'}}\frac{|t-s|^{2H-\varepsilon}}{2^{n\varepsilon}} \ ,
\end{align}
for some proportional constants independent of $n,N_j$.
\end{corollary}
\begin{proof}
For more clarity, set $t_r:=t_r^{N_j}$ ($r=0,\ldots,2^{N_j}$) for the whole proof. Assume that 
\begin{equation}\label{encadr-u-j-v-j}
t_p \leq u_j<t_{p+1} < \ldots < t_q\leq v_j< t_{q+1} \quad \text{with} \ |v_j-u_j|\geq 2^{-N_j} \ ,
\end{equation}
and write
\begin{align}
\der X^{(N_j)}_{u_j v_j}&=\der X^{(N_j)}_{u_j t_{p+1}}+\der X_{t_{p+1} t_{q}}+\der X^{(N_j)}_{t_{q} v_j}\nonumber\\
&=\der X_{t_{p}t_{p+1}}\big(1-2^{-N_j} (u_j-t_p)\big)+\der X_{t_{p+1} t_{q}}+\der X_{t_{q} t_{q+1}}2^{-N_j}(v_j-t_{q}) \ .\label{decomp-appro}
\end{align}
If $t_{q+1}\leq s$, we can apply (\ref{estim-cova-2}) to each of three above terms, which immediately gives (\ref{estim-cova-2-approx}). If $t_{q+1}>s$, we can still apply (\ref{estim-cova-2}) to the first two summands in  (\ref{decomp-appro}), but not to the third one. In this case, let us first write
\begin{align*}
\sum_{i=k}^{\ell-1} \vp\big( \der X^{(N_j)}_{t_{q} t_{q+1}} Y_{2i}\big)^2&=\vp\big( \der X^{(N_j)}_{t_{q} t_{q+1}} Y_{2k}\big)^2+ \sum_{i=k+1}^{\ell-1} \vp\big( \der X^{(N_j)}_{t_{q} t_{q+1}} Y_{2i}\big)^2\\
&\lesssim \big\| \der X^{(N_j)}_{t_{q} t_{q+1}}\big\|^2 \big\|Y_{2k}\big\|^2+\sum_{i=k+1}^{\ell-1} \vp\big( \der X^{(N_j)}_{t_{q} t_{q+1}} Y_{2i}\big)^2\\
&\lesssim |v_j-u_j|^{2H}\frac{|t-s|^{2H-\varepsilon}}{2^{n\varepsilon}} +\sum_{i=k+1}^{\ell-1} \vp\big( \der X^{(N_j)}_{t_{q} t_{q+1}} Y_{2i}\big)^2 \ .
\end{align*}
Due to $N_j\geq N\geq n$, we know that $t_{q+1} <t_{2k+2}^{n+1}$, and we can therefore apply the result of (\ref{estim-cova-2}) to the remaining sum, which entails, as desired,
$$\sum_{i=k+1}^{\ell-1} \vp\big( \der X^{(N_j)}_{t_{q} t_{q+1}} Y_{2i}\big)^2\lesssim \frac{|t-t_{2k+2}^{n+1}|^{2H-\varepsilon}}{2^{n\varepsilon}} 2^{-2HN_j}\lesssim \frac{|t-s|^{2H-\varepsilon}}{2^{n\varepsilon}} |v_j-u_j|^{2H} \ .$$
The case where $|v_j-u_j|\leq 2^{-N_j}$ can then be handled with elementary arguments, and this achieves the proof of (\ref{estim-cova-2-approx}).

\smallskip

The proof of (\ref{estim-cova-2-approx-2}) follows from similar considerations. Observe indeed that, in the situation (\ref{encadr-u-j-v-j}), 
$$\der \{X^{(N_j)}-X\}_{u_j v_j}=\der X_{t_p u_j}-\der X_{t_q v_j}-2^{-N_j}(u_j-t_p) \der X_{t_p t_{p+1}}+2^{-N_j}(v_j-t_q) \der X_{t_q t_{q+1}} \ ,$$
and from here we can apply the above reasoning to each of the summands.

\end{proof}

\section{Non-convergence of the L{\'e}vy area when $H\leq \frac14$}\label{sec:proof-main-result-2}

Let us finally provide the details behind the second main technical result of our study.

\begin{proof}[Proof of Proposition \ref{prop:non-convergence}]
Just as in (\ref{main-term-area}), we have
$$\mathbb{X}^{2,(n+1)}_{01}[1]-\mathbb{X}^{2,(n)}_{01}[1]=\frac12 \sum_{i=0}^{2^n-1} \big\{ Y^{(n)}_{2i}Y^{(n)}_{2i+1}-Y^{(n)}_{2i+1}Y^{(n)}_{2i} \big\} \ ,$$
with $Y^{(n)}_{i}:=\der X_{t_{i}^{n+1}t_{i+1}^{n+1}}$, and so, setting 
$$M^{(n)}:=\vp\big( \big(\mathbb{X}^{2,(n+1)}_{01}[1]-\mathbb{X}^{2,(n)}_{01}[1]\big) \big(\mathbb{X}^{2,(n+1)}_{01}[1]-\mathbb{X}^{2,(n)}_{01}[1] \big)^\ast \big) \ ,$$
we can write
$$
M^{(n)}=\frac14 \sum_{i,j=0}^{2^n-1} \vp\Big( \big\{ Y^{(n)}_{2i}Y^{(n)}_{2i+1}-Y^{(n)}_{2i+1}Y^{(n)}_{2i} \big\} \big\{ Y^{(n)}_{2j+1}Y^{(n)}_{2j}-Y^{(n)}_{2j}Y^{(n)}_{2j+1} \big\} \Big) \ .
$$
Applying formula (\ref{form-wick}) to the semicircular family $\{Y_i\}_{i=0,\ldots ,2^{n+1}-1}$, we can easily expand the latter quantity as
\begin{equation}\label{m-n-delta-n}
M^{(n)}=\frac12 \sum_{i,j=0}^{2^n-1} \Delta^{(n)}(i,j) \ ,
\end{equation}
where
$$\Delta^{(n)}(i,j):=\vp\big(Y^{(n)}_{2i}Y^{(n)}_{2j}\big)\vp\big( Y^{(n)}_{2i+1}Y^{(n)}_{2j+1}\big)-\vp\big(Y^{(n)}_{2i}Y^{(n)}_{2j+1}\big)\vp\big( Y^{(n)}_{2j}Y^{(n)}_{2i+1}\big) \ .$$
At this point, observe that $\Delta^{(n)}(i,j)=\Delta^{(n)}(j,i)$ for all $i,j=0,\ldots,2^n-1$, and for $0\leq i\leq j\leq 2^n-1$, it can be checked that
$$\Delta^{(n)}(i,j) =\frac{1}{4 \cdot 2^{4H(n+1)}} \Gamma_H(2(j-i)) \ ,$$
where, for every $k\geq 0$, $\Gamma_H(k)$ is defined as
\begin{align*}
\Gamma_H(k):=&\big( 2|k|^{2H}-|k+1|^{2H}-|k-1|^{2H}\big)^2\\
&-\big(2|k-1|^{2H}-|k|^{2H}-|k-2|^{2H}\big) \big( 2|k+1|^{2H}-|k|^{2H}-|k+2|^{2H}\big) \ .
\end{align*}
Going back to (\ref{m-n-delta-n}), we thus have
\begin{align}
M^{(n)}&=\frac{1}{8 \cdot 2^{4H(n+1)}} \bigg\{ 2^n \Gamma_H(0)+2\sum_{k=1}^{2^n-1} (2^n-k) \Gamma_H(2k) \bigg\}\nonumber\\
&=\frac{2^{n(1-4H)}}{2^{4H+3}} \bigg\{ \Gamma_H(0)+2\sum_{k=1}^{2^n-1} \bigg(1-\frac{k}{2^n}\bigg) \Gamma_H(2k) \bigg\} \ .\label{m-n-bis}
\end{align}
Denoting by $f_H$ the function at the center of the subsequent Lemma \ref{lem:funct-f-h}, one has, for every $k\geq 1$, $\Gamma_H(k)=-k^{4H} f_H\big( \frac{1}{k}\big)$, and so, using the result of this lemma, we get that for every $k\geq 1$,
$$\big| \Gamma_H(2k)\big| \leq |2k|^{4H}\cdot  \frac{2}{|2k|^4} \leq \frac{1}{4k^3} \ ,$$
which implies that $\Gamma_H(2k) \geq -\frac{1}{4k^3}$. Injecting this lower bound into (\ref{m-n-bis}), we end up with 
$$M^{(n)} \geq \frac{2^{n(1-4H)}}{2^{4H+3}} \bigg\{ \Gamma_H(0)-\frac12\sum_{k=1}^\infty \frac{1}{k^3}+\frac{1}{2^{n+1}} \sum_{k=1}^{2^n-1}\frac{1}{k^2} \bigg\}\geq \frac{2^{n(1-4H)}}{2^{4H+3}} \bigg\{ \Gamma_H(0)-\frac12 \sum_{k=1}^\infty \frac{1}{k^3} \bigg\} \ ,$$
and we can now explicitly check that this lower bound is indeed strictly positive, due to
$$\Gamma_H(0)=2^{2H}\{4-2^{2H}\} \geq 3 \ .$$

\smallskip

As for the second assertion in our statement, it relies of course on estimate (\ref{trace-norm-1}), that is on the bound
$$M^{(n)} \leq \big\| \mathbb{X}^{2,(n+1)}_{01}[1]-\mathbb{X}^{2,(n)}_{01}[1]  \big\|^2 \ .$$
\end{proof}

\begin{lemma}\label{lem:funct-f-h}
Consider the function $f_H$ defined for every $x\in [0,\frac12]$ as
\begin{align*}
&f_H(x):=\\
&\big[2(1-x)^{2H}-1-(1-2x)^{2H}\big] \big[2(1+x)^{2H}-1- (1+2x)^{2H}\big]-\big[ 2-(1+x)^{2H}-(1-x)^{2H}\big]^2 \ .
\end{align*}
Then for all $0<H\leq \frac14$ and $k\geq 1$, it holds that $|f_H(\frac{1}{2k})|\leq \frac{2}{|2k|^4}$.
\end{lemma}

\begin{proof}
Note that the claimed bound follows from rough estimates (and is thus far from optimal). 

\smallskip

First, for $k=1$, let us write $f_H(1/2)$ as
\begin{align*}
f_H(1/2)=&\big[ 2^{1-2H}-1 \big]\bigg[  2H(1-2H)\int_0^{1/2} dy_1 \int_0^{1/2} dy_2 \, (1+(y_1+y_2))^{2H-2} \bigg]\\ 
&\hspace{2cm}-\bigg[2H(1-2H) \int_0^{1/2} dy \int_{-y}^y dz \, (1-z)^{2H-2} \bigg]^2  \ ,
\end{align*}
so that
\begin{align*}
|f_H(1/2)|& \leq \bigg\{ 2H(1-2H) (1/2)^2+(2H(1-2H))^22^{4-4H}\bigg|2\int_0^{1/2} dy \, y\bigg|^2 \bigg\}\\
& \leq \big\{ (1/2)^4+2^4 (1/2)^8\big\}  \leq 2 (1/2)^4 \ .
\end{align*}

\smallskip

Then, for every $x\in [0,\frac14]$, write $f_H(x)$ as
\begin{align*}
f_H(x)=&\big( 2H(1-2H)\big)^2 \bigg\{ \bigg[ \int_0^x dy_1 \int_0^x dy_2 \, (1-(y_1+y_2))^{2H-2} \bigg]\cdot\\ 
&\hspace{1cm}\bigg[ \int_0^x dy_1 \int_0^x dy_2 \, (1+(y_1+y_2))^{2H-2} \bigg]-\bigg[ \int_0^x dy \int_{-y}^y dz \, (1-z)^{2H-2} \bigg]^2 \bigg\} \ ,
\end{align*}
and as a result
$$|f_H(x)| \leq \frac{1}{16}\bigg\{ 2^{2-2H}x^4+\bigg(\frac43\bigg)^{4-4H}\bigg|2\int_0^x dy \, y\bigg|^2 \bigg\} \leq \frac{1}{16}\bigg\{ 4+\bigg(\frac43\bigg)^{4}\bigg\} x^4 \leq x^4 \ .$$
\end{proof}

\end{document}